\documentclass[reqno,a4paper,12pt]{article}
\usepackage{amssymb,amsthm,amsmath,amsfonts}
\usepackage{epsfig}
\usepackage{esint}
\usepackage{cite}

\theoremstyle{plain}
\begingroup
\newtheorem{theorem}{Theorem}[section]

\newtheorem{lemma}[theorem]{Lemma}

\newtheorem{proposition}[theorem]{Proposition}

\endgroup
\theoremstyle{definition}
\newtheorem{definition}[theorem]{Definition}

\newtheorem{remark}[theorem]{Remark}

\newtheorem{example}[theorem]{Example}

\numberwithin{equation}{section}
\numberwithin{figure}{section}

\DeclareMathOperator{\supp}{\mathrm{supp}}

\def\R{{\mathbb R}}
\def\C{{\mathbb C}}

\def\Z{{\mathbb Z}}

\def\ie{i.e.\ }
\def\eg{e.g.\ }

\newcommand\xqed[1]{%
  \leavevmode\unskip\penalty9999 \hbox{}\nobreak\hfill
  \quad\hbox{#1}}
\newcommand\blsquarehere{\xqed{$\blacksquare$}}

\makeatletter
\def\blfootnote{\xdef\@thefnmark{}\@footnotetext}
\makeatother

\title{Partial Balayage on Riemannian Manifolds}

\author{
Bj\"orn Gustafsson,
Joakim Roos \\[.1cm] \small{KTH Royal Institute of Technology}
}

\date{\vspace{.35cm}}

\begin{document}

\maketitle

\begin{abstract}
  A general theory of partial balayage on Riemannian manifolds is
  developed, with emphasis on compact manifolds.  Partial balayage is
  an operation of sweeping measures, or charge distributions, to a
  prescribed density, and it is closely related to (construction of)
  quadrature domains for subharmonic functions, growth processes such as
  Laplacian growth and to weighted equilibrium distributions.

  Several examples are given in the paper, as well as some specific
  results. For instance, it is proved that, in two dimensions,
  harmonic and geodesic balls are the same if and only if the Gaussian
  curvature of the manifold is constant.

\end{abstract}

\vspace{.6cm}


\tableofcontents


\section{Introduction}\label{sec:introduction}

Balayage, in its classical meaning, is a potential theoretic tool for sweeping measures out of a given
domain, where they initially are located, in such a way that the external potentials
do not change. This means that, after sweeping, the measures have to sit on the boundary
of the domain. The idea goes back at least to C.F.~Gauss (see \cite{Doob-1984} for some history) and is inspired by
considerations in electrostatics.

Partial balayage is a generalization of this classical balayage, and amounts
to incomplete sweeping,  to a prescribed density which then will be attained only in an
\emph{a priori} unknown set. Outside this set the potentials are required to remain unchanged.
The first treatments, known to us, of ideas in such a direction are papers and books in geophysics by
D.~Zidarov and collaborators of him. See \cite{Zidarov-1990}, where the terminology
``partial gravi-equivalent mass scattering'' is used. 

Subsequent developments include methods for constructing quadrature domains by M.~Sakai
\cite{Sakai-1979, Sakai-1982, Sakai-1983}, construction of weak solutions to
moving boundary problems for Hele-Shaw flow (in a pure form nowadays often referred to as
Laplacian growth) \cite{Elliott-Janovsky-1981, Gustafsson-1985} and, later on, game theoretic
aspects, ``toppling'', internal diffusion limited aggregation (IDLA) \cite{Diaconis-Fulton-1990, Levine-Peres-2010}, etc.
There are close connections to weighted equilibrium distributions \cite{Saff-Totik-1997},
this theory being in fact essentially equivalent to partial balayage.
Most of the above mentioned developments were initiated independently of each other, and of Zidarov's work.

Partial balayage has so far mainly been considered in Euclidean space, but it is natural, and important
for applications, to extend the theory to Riemannian manifolds. The main purpose of this article is to lay
the ground for such generalizations. We will concentrate on compact manifolds, for which there are
some initial difficulties caused by the fact that only (signed) measures of zero net mass can have globally defined
potentials. The resolution of such difficulties, on the other hand, sheds light also on the theory in the Euclidean case. 

Besides development of a general theory we obtain some specific results and give a number of examples.
One specific result is that, in the case of two dimensions, harmonic balls (defined in terms
of mean-value properties of harmonic functions) are the same as geodesic balls if and only if the Gaussian
curvature of the manifold is constant. In the case of defining partial balayage in bounded regions of Euclidean space 
by using Dirichlet boundary conditions, we also obtain decisive estimates of much mass goes to the boundary.   

The paper is organized as follows. In Section~\ref{sec:hodge} we explain some necessary background material
in Hodge theory and potential theory. Section~\ref{sec:partial balayage} introduces partial balayage on compact Riemannian 
manifolds by using the the physically intuitive idea of energy minimization, first directly in terms of the charge distributions
involved (Definitions~\ref{def:partialbalayage0}) 
and then in terms of the potential $u$  of the difference between the final and initial charge distribution (Definition~\ref{def:partialbalayage}). 
Partial balayage eventually boils down to a free boundary problem, of obstacle type, for this potential $u$.

Section~\ref{sec:variational} contains several equivalent descriptions of partial balayage, for example in terms of a quite
useful complementarity system, and all this ends up with the final definition of partial balayage in terms of an obstacle
problem, in Section~\ref{sec:relaxed} (Definition~\ref{def:partialbalayage2}). This definition is made as general as possible,
but still some assumptions on the charge distributions are needed in order that the balayage shall exist.

In Section~\ref{sec:structure} we show that the result of partial balayage satisfies natural bounds and has the expected
structure, under some rather mild conditions. These are on the other hand are necessary, as  examples in Section~\ref{sec:examples} will show. 
Most of the material in these first six sections is well-known
in the Euclidean case, but the Riemannian manifold setting requires some novel issues to be handled.

The relation between partial balayage and weighted equilibrium distributions,  subharmonic quadrature domains and
Laplacian growth processes are briefly explained  in Sections~\ref{sec:weighted-eq-confs}, \ref{sec:quadrature} and \ref{sec:Laplacian growth}, respectively.
Related to this we show in Section~\ref{sec:balls} the previously mentioned result on geodesic and harmonic balls.

Partial balayage on non-compact Riemannian manifolds is discussed in Section~\ref{sec:nonclosed}. For manifolds with boundary one
has to choose which type of boundary conditions to work with, and there are several possibilities in this respect. In the case of
Dirichlet boundary conditions some of the swept mass will go to the boundary and there arises the question, in case one exhausts
for example $\R^n$ with a sequence of bounded regions, whether this excess mass will eventually be swallowed by the larger
regions or whether it be lost in the limit. This question will be settled in Section~\ref{sec:ball}, and it turns out that the
answer depends on the dimension: in dimension $n\geq 3$ mass may be lost in the limit, as is shown by examples, while in
dimension $n=1,2$ we show by estimates (Theorem~\ref{thm:excess}) that all mass will be caught in the limit.

Section~\ref{sec:examples} contains some relatively elementary examples, with the purpose of illustrating the general theory, and Section~\ref{sec:doubling}
discusses certain symmetric compact manifolds obtained by doubling a manifold with boundary.  

The paper builds on, and is inspired by, many previous papers and books in the area, for example (with an incomplete list, and partly repeating from the beginning
of the introduction) D.~Zidarov \cite{Zidarov-1990}, M.~Sakai \cite{Sakai-1982}, A.~Varchenko,  P.~Etingof \cite{Varchenko-Etingof-1992}, H.S.~Shapiro~\cite{Shapiro-1992},
E.~Saff, V.~Totik \cite{Saff-Totik-1997}, 
H.~Hedenmalm, S.~Shimorin, N.~Makarov \cite{Hedenmalm-Shimorin-2002, Hedenmalm-Makarov-2013}, 
T.~Sj\"odin, S.~Gardiner \cite{Sjodin-2007, Gardiner-Sjodin-2008, Gardiner-Sjodin-2009},
L.~Levine, Y.~Peres \cite{Levine-Peres-2010}, F.~Balogh, J.~Harnad \cite{Balogh-Harnad-2009}.


\newpage
\subsection{Notations and conventions}\label{sec:notations}

\begin{itemize}

\item $B(a,R)=\{x\in \R^n: |x-a|<R\}$, $B_R=B(0,R)$.

\item $S^{n-1}=\partial B(0,1)\subset \R^n$.

\item $*\omega$: the Hodge star acting on a differential form $\omega$.

\item ${\rm vol}^n=*1$: the volume form on an $n$-dimensional Riemannian manifold.

\item $\delta=(-1)^{n(p+1)+1}*d*$: the coexterior derivative ($d$ is the exterior derivative).

\item $\mathcal{L}_{\bf v}$: Lie derivative by a vector field ${\bf v}$.

\item $i(\bf v)\omega$: interior derivation (contraction) of a differential form $\omega$ by a vector field ${\bf v}$ .

\item $\Delta=-(\delta d +d\delta)$: the Hodge Laplacian, with sign chosen so that it agrees with the ordinary 
Laplacian in the Euclidean case. When $u$ is a function we usually work with the $n$-form $d*du =(\Delta u)\,{\rm vol}^n$ instead.

\item $\delta_a$: Dirac current (point charge) at a point $a$, considered as an $n$-form. 

\item $\mathsf{m}(\omega)$: normalized total mass of a differential form $\omega$,
\begin{equation*}
\textsf{m}(\omega)= \frac{1}{{\rm vol}(M)}\int_M\omega. 
\end{equation*}

\item $\mathsf{m}(\varphi)$: normalized total mass of the form
  $\varphi \, \mathrm{vol}^n$ for a function $\varphi$,
\begin{equation*}\label{mvol}
  \textsf{m}(\varphi) = \mathsf{m}(\varphi\, \mathrm{vol}^n) = \frac{1}{{\rm vol}(M)}\int_M\varphi\, {\rm vol}^n.
\end{equation*}
\item $U^\mu$: the Newtonian potential of a (signed) measure $\mu$ in $\R^n$, normalized so that $-\Delta U^\mu =\mu$ and vanishing at infinity
(when $n\geq 3$; logarithmic behaviour when $n=2$).

\item $G^\omega$: the Green's potential of a charge distribution $\omega$ on a compact manifold, satisfying $-d*d G^\omega=\omega - \textsf{m}(\omega) {\rm vol}^n$
and normalized by $\int G^\omega \,{\rm vol}^n =0$.

\item $G(a,b)=G^{\delta_a}(b)$:  Green's kernel on a compact manifold. 

\item $g_M (x,a)$: Dirichlet Green's function on a manifold $M$ with boundary. Defined by $-d*dg_M (\cdot,a)=\delta_a$ in $M$, $g_M(\cdot, a)=0$ on $\partial M$.

\item $\mathcal{E}(\omega_1,\omega_2)= \int dG^{\omega_1}\wedge *dG^{\omega_2}= \int G^{\omega_1}\wedge {\omega_2}$: mutual energy.

\item $\mathcal{E}(\omega)=\mathcal{E}(\omega,\omega)$. 

\item {\it charge distribution} = signed $n$-form current = signed
  measure considered as an $n$-form (on a manifold of dimension $n$).

\item {\it potential} = function (defined a.e.) which locally is the difference between two subharmonic functions
= $\delta$-subharmonic function = function $u$ for which $d*du$ is a charge distribution.

\item ${\rm Bal}(\sigma, \lambda)$: partial balayage of a measure or charge distribution $\sigma$ towards $\lambda$
(see Definitions~\ref{def:partialbalayage0}, \ref{def:partialbalayage} and \ref{def:partialbalayage2}).

\item $W^{1,2}(M)$, $W^{-1,2}(M)_n$, $L^2(M)_p$, etc.: Sobolev and Lebesgue spaces (of functions, $n$-forms, $p$-forms, 
etc.), see Section~\ref{sec:curr-finite-energy}.

\end{itemize}


\section{Hodge theory and potential theory}\label{sec:hodge}

\subsection{Currents of finite energy}
\label{sec:curr-finite-energy}

We shall recall a few concepts from Hodge theory and potential theory. For more details  and general notational conventions we refer
to \cite{Frankel-2012, Warner-1983, Petersen-2006, Berger-2003, Helms-1969, Doob-1984, Armitage-Gardiner-2001}. 
We shall also enter into the terminology of currents (differential forms with distributional coefficients),  see \cite{deRham-1984, Federer-1969}
on this matter. 

We assume that $M$ is a compact (closed) oriented Riemannian manifold of dimension $n\geq 1$. The coexterior derivative $\delta$ is defined on $p$-forms by
$$
\delta=(-1)^{n(p+1)+1}*d* \, ,
$$
where the star is the Hodge star, transforming $p$-forms to complementary $(n-p)$-forms.
Thus $\delta$ takes $p$-forms to $(p-1)$-forms. The Hodge Laplacian is the positive operator
$$
-\Delta= \delta d +d\delta.
$$
A $p$-form $\omega$ is \emph{harmonic} if $\Delta \omega=0$, and on a compact manifold
this is equivalent to that the two equations  $d\omega=0=\delta \omega$ hold. 

The natural inner product on the space of $p$-forms is
\begin{equation}\label{innerproduct}
(\alpha, \beta)_p=\int_M \alpha\wedge *\beta.
\end{equation}
We denote by $L^2(M)_p$ the 
Hilbert space of $p$-forms with $(\omega,\omega)_p<\infty$ and inner product (\ref{innerproduct}).
The Hodge theorem \cite{Warner-1983} says that any  $\omega\in L^2(M)_p$
has an orthogonal decomposition
\begin{equation}\label{hodge}
\omega= \eta +d\alpha+\delta  \beta,
\end{equation}
where $\eta$ is a harmonic $p$-form, $\alpha$ is a coexact $(p-1)$-form,  and $\beta$ is an exact $(p+1)$-form. 
In (\ref{hodge}), the forms $\alpha$ and $\beta$ are not uniquely determined, only $d\alpha$, $\delta\beta$, and $\eta$ are.
The decomposition can however be made more precise as
\begin{equation}\label{hodge1}
\omega= \eta +d\delta\tau+\delta d  \tau=\eta- \Delta \tau,
\end{equation}
where the $p$-form $\tau$ becomes unique on requiring that it shall be orthogonal to all harmonic forms. For this choice of $\tau$
we write
\begin{equation}\label{green}
\tau=G(\omega), 
\end{equation}
with $G$ interpreted as the ``Green's operator'' for solving the Poisson equation $-\Delta \tau =\omega -H(\omega)$.
Here $H$ denotes the orthogonal projection onto the space of harmonic forms (\ie  $H(\omega)=\eta$ in (\ref{hodge1})).

The only harmonic functions on $M$ are the constant functions, and hence the only global  harmonic $n$-forms are 
the constant multiples of the volume form ${\rm vol}^n=*1$. If $\omega$ is any $n$-form, its Hodge decomposition therefore is of the form
$\omega =\textsf{t}\,{\rm vol}^n  - \Delta(\psi \,{\rm vol}^n)$ for some $\textsf{t}\in\R$ and some function $\psi$. We have $\Delta (\psi\,{\rm vol}^n)=(\Delta \psi){\rm vol}^n=d*d\psi$, so this becomes
\begin{equation}\label{hodgeomega}
\omega =\textsf{t}\,{\rm vol}^n  - d* d\psi,
\end{equation}
where
\begin{equation}\label{tvol}
  \textsf{t}=\textsf{m}(\omega) := \frac{1}{{\rm vol}(M)}\int_M\omega. 
\end{equation}
The function $\psi$ becomes uniquely determined on requiring that $\tau=\psi \, \mathrm{vol}^n$
shall be orthogonal to all harmonic $n$-forms, which means that
\begin{equation}\label{psivol}
\int_M \, \psi\, {\rm vol}^n=0.
\end{equation}

With the normalization (\ref{psivol}) of $\psi$ we have, in terms of the Green's operator above,
$$
G(\omega)=\psi\,{\rm vol}^n.
$$
In the right member here only the function $\psi$ carries any information, and we single it out by writing  the same relation also as
\begin{equation}\label{greenpotential}
\psi=G^\omega.
\end{equation}
Thus $G(\omega)=G^\omega {\rm vol}^n$ in general.
We interpret the function $G^\omega$ as a {Green's potential} of $\omega$, for which
\begin{equation}\label{ddG}
-d*d G^\omega= \omega-\textsf{t}\,{\rm vol}^n,
\end{equation}
where $\mathsf{t} = \mathsf{m}(\omega)$ and the term
$\textsf{t}\,{\rm vol}^n =H(\omega)$ shall be interpreted as an
auto\-matic compensating background field balancing the right hand side
to zero net mass. In ordinary potential theory one asks Green's
potentials to vanish on the boundary (or at ``infinity''), but here
there is no boundary, so our normalization will just be the one in
(\ref{psivol}), \ie
\begin{equation}\label{Gvol}
\int_M \, G^\omega \, {\rm vol}^n=0.
\end{equation}

With (\ref{ddG}), (\ref{Gvol}), $G^\omega$ is uniquely determined by $\omega$. However, $\omega$ is not uniquely determined by $G^\omega$. In fact, since
we can always rewrite (\ref{ddG}) as
$$
-d*d G^\omega= \omega + \textsf{s}\,{\rm vol}^n-(\textsf{t}+\textsf{s})\,{\rm vol}^n
$$
we have
\begin{equation}\label{omeganotunique}
G^{\omega + \textsf{s}\,{\rm vol}^n}=G^\omega
\end{equation}
for any $\textsf{s}\in\R$. As a particular case, we notice that $G^{{\rm vol}^n}=0$.

In our applications, the inner product $(f, f)_1$, with $f$  a $1$-form, will have the interpretation of being the energy
of $f$ as a field (like an electric field), or of the $n$-form $d* f$ as a charge distribution.
For a function (``potential'') $u$ we consider the Dirichlet integral $(du,du)_1$ to be its energy.
Thus constant functions have no energy. Similarly, for $n$-forms (``charge distributions'') we consider 
${\rm vol}^n$ to have no energy, and the energy  of $\omega$ in (\ref{hodgeomega}) is then
defined to be the energy $(dG^\omega, dG^\omega)_1$ of its Green's potential $G^\omega$.
For the corresponding quadratic form,  $\mathcal{E}(\omega_1, \omega_2)=(dG^{\omega_1}, dG^{\omega_2})_1$, we get,
after a partial integration and on using (\ref{ddG}), (\ref{Gvol}), 
\begin{equation}\label{Eomega}
\mathcal{E}(\omega_1, \omega_2)=\int_M dG^{\omega_1} \wedge *dG^{\omega_2}=\int_M G^{\omega_1} \wedge \omega_2.
\end{equation}
We shall write also $\mathcal{E}(\omega)=\mathcal{E}(\omega,\omega)$, and note that $\mathcal{E}(\omega)\geq 0$, with equality holding
if and only if $\omega$ is a constant multiple of ${\rm vol}^n$.

The function theoretic interpretations of  forms  having finite energy can be expressed in terms of their belonging to appropriate Sobolev  classes. 
See for example \cite{Schwarz-1995} for careful discussions of Sobolev spaces on Riemannian manifolds.
Specifically, the meaning of $u$, or $du$, having finite energy is that $u$ belongs to the Sobolev space $W^{1,2}(M)$, and an $n$-form
current having finite energy means that it is  in $W^{-1,2}(M)$. Thus $G^\omega \in W^{1,2}(M)$ when $\omega\in W^{-1,2}(M)$.

In some integrals which will come up, the integrand will be a product between two currents, which in general is not well-defined. However,
these currents will in our cases usually belong to two Sobolev spaces which are dual to each other, to the effect that the integral will have 
the interpretation of representing the duality pairing. And when the integral in addition has a measure theoretic interpretation,
the two meanings of the integral will in general agree. See \cite{Brezis-Browder-1979} for some clarification of 
such matters. 

The usual formulas for partial integration (like Stokes' theorem) always hold with appropriate interpretations, essentially 
as a consequence of the  definition of derivatives in the sense of distributions, or currents. For example, if $f$ is a $1$-form 
current of finite energy and $\varphi\in C^\infty (M)$ is a test function, then by definition
$$
\int_M \varphi \wedge {d*f} = -\int_M d\varphi \wedge *f,
$$
and this formula remains valid when $\varphi \in W^{1,2}(M)$.

As a summary, the main spaces and mappings which will show up are exhibited in the sequence
$$
W^{1,2}(M)\stackrel{{d}}{\longrightarrow} L^2(M)_1 \stackrel{\mathrm{*}}{\longrightarrow} L^2(M)_{n-1}\stackrel{{d}}{\longrightarrow} W^{-1,2}(M)_n,
$$
defined by $u\mapsto du\mapsto *du \mapsto d*du$, and where the subscript $p=1, n-1, n$ indicates the degree of the form or current (not written out when $p=0$).
The first map above has a one-dimensional kernel, and the last map a one-dimensional cokernel. As already indicated, the energy of an object at any level in this
sequence is defined as the squared norm of it when it is moved to $L^2(M)_1$ by the above mappings, and constant functions and constant multiples of the volume form
are then left without energy. Similarly, the mutual energy between two objects is related in
the same way to the inner product in $L^2(M)_1$ (which is defined by (\ref{innerproduct}) with $p=1$).


\subsection{Charge distributions and potentials}\label{sec:charge}

We will need to go beyond the finite energy setting of currents described above. On the other hand we shall also restrict a little, because we shall only deal with $n$-form
currents which can be considered as signed measures, \ie those which are differences between two positive $n$-form currents. 

A positive measure is naturally associated to a positive functional,
\begin{equation}\label{L}
L:C^\infty (M)\to \R, \quad \text{with\,\,} L(\varphi)\geq 0 \text{\,\,for\,\,} \varphi\geq 0,
\end{equation}
and this can be thought of as  an $n$-form current $\omega$, \ie an $n$-form with distributional coefficients,  by writing (formally)
\begin{equation}\label{integral}
L(\varphi)=\int_M \varphi \wedge\omega \quad (\varphi\in C^\infty (M)).
\end{equation}

As an example, the Dirac measure (point mass) at a point $a\in M$ corresponds to the functional $L(\varphi)=\varphi(a)$ and, notationally, $\delta_a$ will then refer
to the corresponding $n$-form current.
A signed measure is the difference between two positive measures, and the corresponding current will then be called a signed $n$-form current or, for briefness,
a {\it charge distribution}. Any function, or $0$-current, $\psi$ for which $d*d\psi$ is a charge distribution will be called a {\it potential}.

\begin{remark}
  What we call potentials have by S.~Gardiner, T.~Sj\"odin
  \cite{Gardiner-Sjodin-2012, Gardiner-Sjodin-2014} been named
  $\delta$-subharmonic functions (to be interpreted as the difference
  between two subharmonic functions).
  \blsquarehere
\end{remark}

If $\omega$ is a positive charge distribution, then its Green's potential $G^\omega$, which so far has been considered just as a $0$-current (or a function defined a.e.),  
has a canonical representative  in form of a lower semicontinuous function with values in $\R\cup \{+\infty\}$. This is also the largest lower semicontinuous representative of 
$G^\omega$.
Any charge distribution $\omega$ has a minimal decomposition as a difference between two nonnegative charge distributions,  namely the Jordan decomposition 
$\omega=\omega_+  -\omega_-$.  If a $0$-current $\psi$ satisfies $-d*d\psi =\omega$ then, necessarily, 
\begin{equation}\label{canonical}
\psi=G^{\omega_+}-G^{\omega_-}+c   
\end{equation}
(in the sense of currents) for some constant $c$. So this is the general form of a potential in $M$.
There may be a small set of points (of capacity zero) where $\psi$ cannot be assigned any particular value,
because both Green's potentials above may take the value $+\infty$ at the same point. Thus, in general, a potential is defined only a.e. (or, more precisely, quasi everywhere)
when considered as a function. However, if $\psi$ is bounded either from above or from below then at most one of the Green's potentials can attain the value $+\infty$,
hence $\psi$ will in this case have have a canonical representative as an everywhere defined function, namely that function given by the right member in (\ref{canonical}).

Since a charge distribution can always be decomposed into its positive and negative parts one can in many cases define  the mutual energy $\mathcal{E}(\omega_1, \omega_2)$
even in cases when one or both of the individual charge distributions $\omega_1$, $\omega_2$ have infinite energy. In fact, the mutual energy can always be decomposed as 
\begin{align}
  \mathcal{E}(\omega_1, \omega_2) =& \ \mathcal{E}((\omega_1)_+, (\omega_2)_+)-\mathcal{E}((\omega_1)_+, (\omega_2)_-) \nonumber \\
  &-\mathcal{E}((\omega_1)_-,( \omega_2)_+)+\mathcal{E}((\omega_1)_-, (\omega_2)_-), \label{Eomega12}
\end{align}
where each individual term is either $+\infty$ or a finite nonnegative number. Thus the the mutual energy has a definite meaning 
as long as no two terms of opposite signs in (\ref{Eomega12}) are infinite.

For example, the Dirac current $\delta_a$ certainly has infinite energy, but if $a\ne b$, then $\mathcal{E}(\delta_a, \delta_b)$ is still finite and has a natural interpretation:
it is the {\it Green's kernel} $G(a,b)$, which can be defined as
\begin{equation}\label{Gab}
G(a,b)=G^{\delta_a}(b).
\end{equation}
In fact, using (\ref{ddG}) and (\ref{Gvol}) for $\omega=\delta_a, \delta_b$ we have
\begin{align*}
  \mathcal{E}(\delta_a, \delta_b) &= \int_M dG^{\delta_a}\wedge *dG^{\delta_b}= -\int_M G^{\delta_a}\wedge d*dG^{\delta_b} \\
                                  &= \int_M G^{\delta_a}\wedge \left( \delta_b-\mathsf{m}(\delta_b)\,{\rm vol}^n \right) = G^{\delta_a}(b) = G(a,b), 
\end{align*}
which in addition shows that $G(a,b)$ is symmetric. Note that the Dirichlet integral above is absolutely integrable because the singularity of $dG^{\delta_a}$
is relatively mild (namely  like $|x-a|^{1-n}$), and that the use of Stokes' formula can easily be justified by classical methods (\eg by performing partial
integration after having removed small balls around the singularities).

One advantage with $G(a,b)$ is that it allows for an expression for the energy as a double integral with a kernel. We have
\begin{equation}\label{GG}
G^{\omega}(x)=\int_M G(x,y)\, \omega(y),
\end{equation}
and so
\begin{equation}\label{EG}
\mathcal{E}(\omega_1,\omega_2)=\int_M \int_M G(x,y)\, \omega_1(x) \otimes\omega_2 (y).
\end{equation}
Slightly more generally, any potential $\varphi$ can be represented as
\begin{equation}\label{varphiG}
\varphi(x) =\textsf{m}(\varphi) \,{\rm vol}^n-\int_M G(x,y)\, d*d\varphi(y).
\end{equation}
The kernel representation (\ref{EG}) of the energy can be used sometimes to replace a use of Stokes' formula by an application of Fubini's theorem, 
which may be considered as more ``robust''. 

An example is the formula
\begin{equation}\label{varphi12}
\int_M \varphi_1\, d*d\varphi_2 = \int_M \varphi_2\, d*d\varphi_1,
\end{equation}
which should be true under general circumstances. If the $\varphi_j$ have finite energy it is true (at least if the integrals are
interpreted as duality pairings between $W^{1,2}(M)$ and $W^{-1,2}(M)_n$), but if they are just potentials it need not be true.
Indeed, $d*d\varphi_2$ may for example contain a point mass, while it need not be possible to assign any particular value to
$\varphi_1$ at the location of this mass, so even the meaning of the integral is in general ambiguous. However, if each of
$\varphi_j$ ($j=1,2$) is bounded either from below or from above then (\ref{varphi12}) does hold, provided the $\varphi_j$ are
defined pointwise as in (\ref{canonical}) with the canonical lower semicontinuous representatives of $G^{\omega_+}$ and
$G^{\omega_-}$ used.

Assume for example that $\varphi_1\leq C_1<\infty$ and that $\varphi_2\geq C_2>-\infty$. Set $\omega_j=-d*d\varphi_j$.
For $\varphi_1$ we conclude from (\ref{canonical}) (for $\psi=\varphi_1$) that $G^{(\omega_1)_+}$ is bounded form above, hence that $(\omega_1)_+$ has finite
energy:
$$
\mathcal{E}((\omega_1)_+)= \int_M dG^{(\omega_1)_+}\wedge *dG^{(\omega_1)_+}= \int_M G^{(\omega_1)_+} \wedge {(\omega_1)_+}<\infty.
$$
Similarly it follows that $(\omega_2)_-$ has finite energy. Now  the left member of (\ref{varphi12}) becomes
\begin{align*}
  \int_M \varphi_1 d*d\varphi_2 =& -\int_M (G^{(\omega_1)_+} - G^{(\omega_1)_-} +c )\wedge ({(\omega_2)_+} - {(\omega_2)_-}) \\
                                =& -\mathcal{E}((\omega_1)_+, (\omega_2)_+)+\mathcal{E}((\omega_1)_+, (\omega_2)_-) \\
                                &+\mathcal{E}((\omega_1)_-, (\omega_2)_+)-\mathcal{E}((\omega_1)_-, (\omega_2)_-).  
\end{align*}
The only term here which may be infinite is the third one, $\mathcal{E}((\omega_1)_-, (\omega_2)_+)$, because every other term contains, after partial integration,
a Green's potential which is bounded either from above or below, which guarantees finiteness. 

In conclusion, $\int_M \varphi_1 \,d*d\varphi_2$ has a definite value, which may be $+\infty$. And in addition, the expression for it in terms of energies is symmetric
under changes $1\leftrightarrow 2$, which means that it must equal $\int_M \varphi_2 \, d*d\varphi_1$. This proves (\ref{varphi12}) under the stated assumptions.


\section{Partial balayage by energy minimization}\label{sec:partial balayage}

Partial balayage is a sweeping operation which depends on a  measure $\lambda$, which is kept fixed, 
and then sweeps any measure $\sigma$ to a measure $\nu$ which satisfies $\nu\leq \lambda$ everywhere by using a minimum amount of work, \ie so that
the energy of $\nu-\sigma$ is minimized. This turns out to entail  that the Newtonian potentials
of $\sigma$ and $\nu$ agree on the set where $\nu<\lambda$. Partial balayage can be defined on arbitrary Riemannian manifolds, but in the literature
full details have so far been given only for cases of subdomains of $\R^n$.
Now we are ready to define partial balayage in the finite energy setting.

\begin{definition}\label{def:partialbalayage0}
Let $\sigma$ and $\lambda$ be two charge distributions of finite energy and satisfying
\begin{equation}\label{sigmalambda}
\int_M \sigma\leq \int_M \lambda.
\end{equation}
Then there is a unique charge distribution $\nu$ which solves the minimum energy problem
\begin{equation}\label{minimumenergy}
\min_\nu \mathcal{E}(\nu-\sigma): \quad \nu\leq \lambda, \quad \int_M \nu=\int_M \sigma.
\end{equation}
We call $\nu$ the {\it partial balayage} of $\sigma$ to $\lambda$ and write
$$
{\rm Bal\,}(\sigma,\lambda)=\nu.
$$
\end{definition}
Above one thinks of $\lambda$ as being fixed, so the sweeping operation is really the replacement $\sigma\mapsto \nu$,
and the last side condition in (\ref{minimumenergy}) says that the total mass shall not be changed under this operation.
This is in fact a necessary requirement in order that the solution of the minimization problem shall be unique, because  
$\mathcal{E}$ not being positive definite means that one could otherwise add an arbitrary multiple of ${\rm vol}^n$ to $\nu$.
As for the existence, it is  clear that (\ref{sigmalambda}) is the only assumption needed to ensure that the set of $\nu$ to be minimized over is not empty,
hence the minimization problem has indeed a (unique) solution.

It is immediate from the definition that 
\begin{equation}\label{covariance}
{\rm Bal\,}(\sigma+\tau,\lambda+\tau)={\rm Bal\,}(\sigma,\lambda)+\tau
\end{equation}
for any $\tau$, hence it will be enough to deal with ${\rm Bal\,}(\sigma,0)$, as far as the theoretical studies concern.  
The general case can then be recovered by 
\begin{equation}\label{balgeneral}
{\rm Bal\,}(\sigma,\lambda)={\rm Bal\,}(\sigma-\lambda,0)+\lambda.
\end{equation}

Thus we assume now 
\begin{equation}\label{sigmanegative}
\int_M \sigma \leq 0,
\end{equation}
and then $\nu={\rm Bal\,}(\sigma,0)$ is defined as the solution of 
\begin{equation}\label{minimumenergy1}
\min_\nu \mathcal{E}(\nu-\sigma): \quad \nu\leq 0, \quad \int_M \nu=\int_M\sigma.
\end{equation}
To spell this out in terms of potentials we make Hodge decompositions of $\sigma$ and $\nu$. Due to the last side condition in (\ref{minimumenergy1})
these will, after a sign change,  be of the form
\begin{equation}\label{hodgesigma}
\sigma = - \textsf{t}\,{\rm vol}^n +d*d\psi,
\end{equation} 
\begin{equation}\label{hodgenu}
\nu = - \textsf{t}\,{\rm vol}^n +d*dv,
\end{equation} 
with the same $\textsf{t}\in\R$, and for some functions $\psi$ and $v$. The value of $\textsf{t}$ is 
\begin{equation}\label{t}
\textsf{t}=-\textsf{m}(\sigma)= -\frac{1}{{\rm vol}^n(M)}\int_M\sigma\geq 0 
\end{equation}
and $\psi$ may be normalized by (\ref{psivol}), which then gives
\begin{equation}\label{psiGsigma}
\psi=-G^\sigma.
\end{equation}
We shall not normalize $v$ in the same way, however, it will become 
fixed after we have normalized the difference function
\begin{equation}\label{uvpsi}
u=v-\psi
\end{equation}
in a special way (see (\ref{complementarity}) below).

Taking the difference between the two Hodge decompositions above gives $\nu-\sigma =d*du$. From this we see that
$$
\mathcal{E}(\nu-\sigma)= \int_M du \wedge * du.
$$
The definition of partial balayage therefore boils down to the following.
\begin{definition}\label{def:partialbalayage} (Reformulation of Definition~\ref{def:partialbalayage0} in the case $\lambda=0$.)
Under the assumption (\ref{sigmanegative}) and when $\sigma$ has finite energy,
\begin{equation}\label{partialbalayageu}
{\rm Bal\,}(\sigma,0)=d*du+\sigma,
\end{equation}
where $du$ is the unique solution of
\begin{equation}\label{minimizationu}
\min \int_M du\wedge *du: \quad 
d *du+\sigma \leq 0
\quad(du\in L^2(M)_1).
\end{equation}
\end{definition}

The above definition only refers to the $1$-form $du$, required then to be exact, but in the next section
we shall find a convenient way of adjusting the free additive constant in $u$.


\section{Variational and complementarity  formulations}\label{sec:variational}

If we have equality  in (\ref{sigmanegative}), then, in the side condition for (\ref{minimizationu}), $du$ can be chosen to satisfy $-d*du=\sigma$, 
and there will be no other choice of $du$ which satisfies $d*du+\sigma\leq 0$. 
It follows that  ${\rm Bal\,}(\sigma,0)=0$ when $\int_M \sigma=0$. 
We need not further discuss this case, so we assume now that 
\begin{equation}\label{sigmastrictlynegative}
\int_M \sigma < 0.
\end{equation}
Then also
$$
\int_M d*du+\sigma <0,
$$
and it follows that we can fix the free additive constant in  $u$ by requiring that
\begin{equation}\label{complementarity}
\int_M u\wedge (d*du+\sigma)= 0.
\end{equation}

Next we turn to the variational formulation of (\ref{minimizationu}). After a partial integration this becomes,
\begin{equation}\label{variation}
\int_M u\wedge d*d(u-\varphi)\geq 0 \quad \text{for all\,\,} \varphi \text{ satisfying\,\, }d*d\varphi +\sigma \leq 0.
\end{equation}
Under the variations allowed here, the $n$-form
\begin{equation}\label{tau}
\tau= -(d*d\varphi +\sigma)
\end{equation}
is subject only to the constraint $\tau\geq 0$ and that the total mass of $\tau$ is fixed by
\begin{equation}\label{s}
\int_M \tau= -\int_M \sigma>0.
\end{equation} 
Given any $\tau$ satisfying these constraints there exists (by the Hodge theorem) a solution $\varphi$
of (\ref{tau}). Thus, in terms of $\tau$  and on using
(\ref{complementarity}), the variational inequality (\ref{variation}) can be written in the form
$$
\int_M u\wedge \tau \geq 0 \quad \text{for all\,\,} \tau\geq 0 \text{ satisfying } \int_M \tau =-\int_M \sigma.
$$
Clearly this makes it impossible for $u$ to attain negative values (on a set of positive measure),
hence we conclude that $u\geq 0$. Thus we have proved everything but the uniqueness statement in
the following theorem.

\begin{theorem}\label{thm:complementarity}
Assuming (\ref{sigmastrictlynegative}), the linear complementarity problem
\begin{equation}\label{complementaritysystem}
\begin{cases}
u\geq 0,\\[.2cm]
d*du+\sigma\leq 0,\\[.2cm]
\displaystyle \int_M u\wedge (d*du+\sigma)= 0.
\end{cases}
\end{equation}
has a unique solution $u$, and its exterior derivative $du$ is the unique minimizer of (\ref{minimizationu}).
\end{theorem}

The  uniqueness follows from the observation that starting from (\ref{complementaritysystem})
one can go backwards through the above steps to arrive at the variational formulation (\ref{variation}),
which is equivalent to the minimization problem (\ref{minimizationu}) and hence has a unique solution (first for $du$, after which $u$ itself gets determined by (\ref{complementarity})).

Reformulated in terms of the function $v$ in (\ref{hodgenu}), the system (\ref{complementaritysystem}) takes the form
\begin{equation}\label{complementaritysystem1}
\begin{cases}
v\geq \psi,\\[.2cm]
d*dv\leq \textsf{t}\,{\rm vol}^n,\\[.2cm]
\displaystyle \int_M (v-\psi)\wedge (d*dv-\textsf{t}\, {\rm vol}^n)= 0,
\end{cases}
\end{equation}
and the result of the balayage then is
\begin{equation}\label{partialbalayagev}
{\rm Bal\,}(\sigma,0)=d*dv-\textsf{t}\, {\rm vol}^n.
\end{equation}

The functions $u$ and $v$ are Green's potentials up to additive constants:
\begin{equation}\label{uvG}
\begin{cases}
u=G^{\sigma-{\rm Bal}(\sigma,0)}+C,\\[.2cm]
v=-G^{{\rm Bal}(\sigma,0)}+C.
\end{cases}
\end{equation}
Since ${\rm Bal}(\sigma,0)\leq 0$, $v$ is lower semicontinuous with values in $\R\cup\{+\infty \}$, in particular $v$ is bounded from below.  
The constant $C$ (the same in both equations) is obtained from the normalization (\ref{complementarity}) of $u$, and it depends in a nonlinear way on $\sigma$.

The latter way, (\ref{complementaritysystem1}), of writing the complementarity system connects to the classical obstacle problem
(see \cite{Kinderlehrer-Stampacchia-1980, Friedman-1982}, for example), \ie the problem of finding
the smallest superharmonic function passing above a given obstacle, here represented by
$\psi$. The only difference in our case is that the solution $v$ is not really required to be superharmonic,
only to satisfy $\Delta v\leq\textsf{t}$, where $\textsf{t}>0$. 
Similarly, the function $u$ can be characterized as the smallest function $u$ satisfying $u\geq 0$ and $d*d u+\sigma \leq 0$.

Instead of asking for the smallest superharmonic   (in the classical case) function
passing above the obstacle one may ask for a function of smallest Dirichlet norm, \ie energy, passing the obstacle.
This will give the same solution. Expressed in terms of $u$ this gives  (in our case) a minimization problem which in a sense is dual to the previous one 
(\ref{minimizationu}) and is equivalent to it. It is
\begin{equation}\label{minu}
\min \int_M (du\wedge *du -2u\wedge \sigma): \quad  u\geq 0 \quad (du\in L^2(M)_1).
\end{equation}
One easily checks that the variational formulation of this problem leads to (\ref{complementaritysystem}).
Expressed in terms of $v$, (\ref{minu}) becomes
\begin{equation}\label{minv}
\min \int_M (dv\wedge *dv +2\textsf{t}v\, {\rm vol}^n): \quad  v\geq \psi \quad (dv\in L^2(M)_1).
\end{equation}

In summary, in the above set-up with $\sigma$ having finite energy, the problems (\ref{minimumenergy}), (\ref{minimizationu}), (\ref{complementaritysystem}), (\ref{complementaritysystem1}),
(\ref{minu}), (\ref{minv}) are all equivalent, and they are also equivalent to the obstacle problem of finding the smallest function $v$ satisfying
$\Delta v\leq \textsf{t}$, $v\geq \psi$, or the smallest $u$ satisfying $u\geq 0$, $d*du+\sigma\leq 0$. 
An additional equivalent minimization problem is that obtained by reformulating (\ref{minimizationu}) in terms of $v$, namely
\begin{equation}\label{minvsigma}
\min \int_M (dv\wedge *dv +2v \wedge \sigma+2\textsf{t}v\, {\rm vol}^n): \quad \Delta  v\leq \textsf{t} \quad (dv\in L^2(M)_1).
\end{equation}
The above statements will be extended and be made more precise in the next section.


\section{Partial balayage by the obstacle problem}\label{sec:relaxed}

It is important to allow point masses (Dirac measures) in the theory of partial balayage, but unfortunately these do
not have finite energy (in dimension $n\geq 2$). Therefore a more general definition of partial balayage is desirable.
It turns out that it is possible to allow $\sigma_+$ to be a completely arbitrary positive charge distribution,
but some assumptions on $\sigma_-$ are necessary, as examples in Section~\ref{sec:examples} will show. 

At first we shall assume that $\sigma$ is a charge distribution such that $\sigma_-$ has finite energy. 
With thus $\sigma_+$ allowed to be a completely arbitrary positive charge distribution there need not exist any $du$ of finite energy satisfying the side condition in
(\ref{minimizationu}), and the functional to be minimized in (\ref{minu}) need not be bounded from below.
Hence these minimization problems do not 
make sense in the present generality. However, it turns out that it is possible to work with minimization problems formulated in terms of $v$, for example (\ref{minv}).
For this we define $\textsf{t}$ by (\ref{t}) and then the potential $\psi$ by (\ref{hodgesigma}) (or (\ref{psiGsigma}) after normalization), which we reproduce as
\begin{equation}\label{psirelaxed}
d*d\psi=\sigma+\textsf{t}\,{\rm vol}^n.
\end{equation}
As the right member of (\ref{psirelaxed}) has zero net mass it follows from the general theory of metric differential equations
\cite{deRham-1984} that a potential $\psi$ satisfying (\ref{psirelaxed}) indeed exists. It becomes uniquely determined on
demanding that (\ref{psivol}) holds, but it will not necessarily have finite energy anymore.

In order that  the minimization problem (\ref{minv}) shall be useful we need to make sure that there exists at least one competing function $v$ for which the functional to be minimized
is finite. To that end, set
\begin{equation}\label{psipm}
\begin{cases}
\psi_{(+)}=-G^{\sigma_+},\\
\psi_{(-)}=-G^{\sigma_-}.
\end{cases}
\end{equation}
Then the $\psi_{(\pm)}$ are upper semicontinuous with values in $\R\cup \{-\infty\}$, in particular they are bounded from above.
With $C<\infty$ being an upper bound for $\psi_{(+)}$ we have
$$
\psi=\psi_{(+)}- \psi_{(-)}\leq C-\psi_{(-)},
$$
and it follows that the right member here is a competing function of finite energy for (\ref{minv}), as desired.

We conclude that the constraint set for (\ref{minv}), namely
\begin{equation}\label{K}
\mathcal{K}=\{v\in W^{1,2}(M): v\geq \psi\},
\end{equation}
is a non-empty closed convex cone in $W^{1,2}(M)$, and standard Hilbert space theory then ensures the existence and uniqueness of minimizer $v$
of (\ref{minv}).  Notice that the problematic features of $\sigma$ and $\psi$ (that $\sigma_+$ need not have finite energy and that  $\psi$ need not be in
$W^{1,2}(M)$) are now hidden in the set $\mathcal{K}$ and so do not cause any problems (we only had to work a little to show that $\mathcal{K}$ was non-empty).
We have now proved most of the following theorem.

\begin{theorem}\label{thm:partialbalayage1}
Let $\sigma$ be a charge distribution satisfying (\ref{sigmastrictlynegative}) and such that $\sigma_-$ has finite energy, and  let $\mathsf{t}>0$ and $\psi$
be defined by (\ref{t}) and (\ref{psirelaxed}), respectively. 
Then there is a unique minimizer $v$ of the functional
$$
J(v)= \int_M (dv\wedge *dv +2\mathsf{t}\,v\, {\rm vol}^n) \quad (v\in \mathcal{K}),
$$
and it satisfies  the complementarity system (\ref{complementaritysystem1}), in particular $\Delta v\leq \mathsf{t}$.
In the case that also $\sigma_+$ has finite energy, so that Definition~\ref{def:partialbalayage} applies, we have
\begin{equation}\label{partialbalayageprel}
{\rm Bal\,}(\sigma,0)= d*d v -\mathsf{t}\,\,{\rm vol}^n=(\Delta v- \mathsf{t})\,{\rm vol}^n.
\end{equation}
\end{theorem}

\begin{proof}
We just remark that the complementarity system (\ref{complementaritysystem1}) follows in a straight-forward fashion from the variational formulation of the minimization problem.
Everything else has already been proved. 
\end{proof}

Thus we could use the minimizer of $J$ and equation (\ref{partialbalayageprel}) as a way to extend the definition of partial balayage. However, we shall go a little bit further before
extending the definition, but Theorem~\ref{thm:partialbalayage1} will be an important ingredient.
In fact, the minimizer in Theorem~\ref{thm:partialbalayage1} can be alternatively characterized as
the smallest function $v\in \mathcal{K}$ satisfying $\Delta v\leq \textsf{t}$ (see more precisely Theorem~\ref{thm:consistence} below), 
and in this characterization we shall simply remove the requirement $v\in W^{1,2}(M)$
(appearing in the definition of $\mathcal{K}$). Accordingly, we now also abandon the assumption that $\sigma_-$ has finite energy. 
In addition, we shall allow the equality case for (\ref{sigmanegative}) 
in the discussion below.

We then end up with the following, final, definition of partial balayage.

\begin{definition}\label{def:partialbalayage2} 
Given any charge distribution $\sigma$,  let $\textsf{t}$ and $\psi$ be defined by (\ref{t}) and (\ref{psirelaxed}). 
Assume that there exists at least one function $v$ which, considered as a distribution or $0$-current, satisfies   
\begin{equation}\label{assumptionobstacleproblem}
\begin{cases}
\phantom{\Delta} v\geq \psi,\\[.2cm]
\Delta v\leq \textsf{t}.
\end{cases}
\end{equation}
Then there is a smallest such $v$, and in terms of this we define
\begin{equation}\label{partialbalayage} 
{\rm Bal\,}(\sigma,0)= (\Delta v -\textsf{t})\,{\rm vol}^n.
\end{equation}
The more general version ${\rm Bal\,}(\sigma,\lambda)$ is thereafter defined by (\ref{balgeneral}).
\end{definition}

\begin{remark}
We did not assume (explicitly) that (\ref{sigmanegative})  holds, but this is equivalent to $\textsf{t}\geq 0$,
which is contained in the assumption (\ref{assumptionobstacleproblem}). To be precise, if $\textsf{t}<0$ then there exists no $v$ at all satisfying  (\ref{assumptionobstacleproblem}), 
and if $\textsf{t}=0$ then only constant functions $v$ can satisfy  (\ref{assumptionobstacleproblem}). In the $\textsf{t}=0$ case, ${\rm Bal\,}(\sigma,0)$ exists and $=0$ if $\psi$ is
bounded from above, otherwise ${\rm Bal\,}(\sigma,0)$ does not exist.
\blsquarehere
\end{remark}

\begin{example}\label{ex:psiuppersemicontinuous}
Assume that (\ref{sigmanegative}) holds and that $G^{\sigma_-}$ is a continuous function (or even just that it is bounded from above).
Then ${\rm Bal\,}(\sigma,0)$ exists because $\psi=G^{\sigma_-}-G^{\sigma_+}$ will be bounded from above and $\textsf{t}\geq 0$, so a large constant function $v$ will do
in (\ref{assumptionobstacleproblem}).

More generally, ${\rm Bal\,}(\sigma,0)$ exists whenever $\sigma_-$ has finite energy, because in this case the cone $\mathcal{K}$ is nonempty (as was shown before
Theorem~\ref{thm:partialbalayage1}), and when $\mathcal{K}$ is nonempty a function $v$ satisfying (\ref{assumptionobstacleproblem}) is obtained by solving the 
minimization problem in Theorem~\ref{thm:partialbalayage1}. See more precisely Theorem~\ref{thm:consistence} below.
\end{example}

\begin{remark}\label{rem:Gcontinuous}
As to continuity properties of the three functions in (\ref{uvpsi}),  $v$ is always lower semicontinuous (\ie has such a representative), 
as is clear from (\ref{uvG}), but neither $\psi$ nor $u$ need to be semicontinuous in any direction.
However, if  $G^{\sigma_-}$ is assumed to be continuous, then $\psi$ is upper semicontinuous and $u$ lower semicontinuous (by (\ref{psiGsigma} and  (\ref{uvG})). 
A less obvious fact is that $v$ is fully  continuous (not only lower semicontinuous) when $G^{\sigma_-}$ is continuous.  See Lemma~3 in  \cite{Gardiner-Sjodin-2009} for the proof.
(The conclusion  also follows from Lemma~\ref{lem:bounds} below, but the proof of that lemma  uses the result in \cite{Gardiner-Sjodin-2009}.) \\ \blsquarehere
\end{remark}

Now we want to show explicitly that Definition~\ref{def:partialbalayage2} is consistent with the previous definitions of partial balayage. This is stated in
Theorem~\ref{thm:consistence} below, but first we need a lemma.

\begin{lemma}\label{lem:J}
Let $v_1$, $v_2$ be functions (or distributions) satisfying $\Delta v_j\leq \mathsf{t}$ and  $v_1\leq v_2$. Then
$$
J(v_1)\leq J(v_2).
$$
\end{lemma} 

\begin{proof}
Assume first that $v_1$, $v_2$ have finite energy. Then
\begin{align*}
  J(v_1) &= \int_M (dv_1\wedge *dv_1 +2\textsf{t}\,v_1)\, {\rm vol}^n= \int_M (-v_1\wedge d*dv_1 +2\textsf{t}\,v_1)\, {\rm vol}^n
  \\
         &= \int_M (v_1(\textsf{t}-\Delta v_1 +\textsf{t}\,v_1)\, {\rm vol}^n\leq \int_M (v_2(\textsf{t}-\Delta v_1) +\textsf{t}\,v_1)\,
           {\rm vol}^n \\
         &= \int_M (v_1(\textsf{t}-\Delta v_2) +\textsf{t}\,v_2)\, {\rm vol}^n\leq  \int_M (v_2(\textsf{t}-\Delta v_2)
           +\textsf{t}\,v_2)\, {\rm vol}^n \\
         &= \int_M (dv_2\wedge *dv_2 +2\textsf{t}v_2\, {\rm vol}^n)=J(v_2).
\end{align*}
When $v_1$, $v_2$ are allowed to have infinite energy, the in\-equality ${J(v_1)\leq J(v_2)}$ remains valid, with $+\infty$ as possible values.
In fact, if $C$ is any constant then $\min(v_j,C)$ ($j=1,2$) have finite energy and satisfy the remaining assumptions above, so we obtain
$J(\min(v_1,C))\leq J(\min(v_2,C))$. Letting here $C\to +\infty$ gives the desired inequality, by monotone convergence.
\end{proof} 

\begin{theorem}\label{thm:consistence}
Assume, for a given charge distribution $\sigma$,  that ${\rm Bal\,}(\sigma,0)$ exists as in Definition~\ref{def:partialbalayage2}, and let $v$ be
the function in (\ref{partialbalayage}). Then there are two possibilities:
\begin{itemize}

\item[(i)] $v\in W^{1,2}(M)$. In this case  $v\in\mathcal{K}$ and $v$ is the unique minimizer of $J$ in Theorem~\ref{thm:partialbalayage1}.

\item[(ii)] $v\notin W^{1,2}(M)$. In this case the cone $\mathcal{K}$ is empty and there is no minimizer at all of $J$.

\end{itemize}
If $\sigma$ has finite energy (\ie  $\psi\in W^{1,2}(M)$), then the first case  $(i)$ above applies and the new definition (Definition~\ref{def:partialbalayage2}) of partial balayage agrees with the old one
(Definition~\ref{def:partialbalayage0} or Definition~\ref{def:partialbalayage}). 
\end{theorem}

\begin{proof}
Assume that $\mathcal{K}$ is nonempty and let $w\in\mathcal{K}$ be the minimizer of $J$.
Then   $\Delta w\leq \textsf{t}$ by Theorem~\ref{thm:partialbalayage1}, and since also $\Delta v\leq \textsf{t}$ it follows that $\Delta\min(v,w)\leq \textsf{t}$. 
Similarly, $\psi\leq \min(v,w)$. Since $v$ was the smallest function with these properties it follows that $v\leq \min(v,w)$, \ie $v\leq w$.

Now Lemma~\ref{lem:J} shows that $J(v)\leq J(w)$, hence that 
$v=w$, since the minimizer of $J$ is unique. All statements of  the theorem now follow.
\end{proof}

In the setting of Definition~\ref{def:partialbalayage2}, the function $u=v-\psi$ still exists, even though it need not have finite energy,
and (\ref{partialbalayage}) becomes (\ref{partialbalayageu}) when expressed in terms of $u$.
The two complementarity systems (\ref{complementaritysystem}), (\ref{complementaritysystem1}) (which are equivalent)
remain valid, but need some reinterpretation. 
Previously they arose from the variational formulations of minimum norm problems in a Hilbert space, and they more exactly express that the function involved,
$u$ or $v$, is the result of an orthogonal projection onto a convex cone (for example $\mathcal{K}$). 
Setting 
\begin{equation}\label{muusigma}
\mu=-(d*du+\sigma)=-{\rm Bal}(\sigma,0)
\end{equation} 
and writing the complementarity system  (\ref{complementaritysystem}) (for example) as 
\begin{equation}\label{umu}
u\geq 0,\quad\mu\geq 0, \quad \int_M u\wedge \mu=0,
\end{equation}
the last  identity exactly expresses the orthogonality.

For the approach taken in Definition~\ref{def:partialbalayage2} there is another kind of variational formulation, namely saying that wherever the solution $v$
is not in contact with the obstacle $\psi$, it could have been made smaller, by a Poisson kind modification, unless $\Delta v$ is already at its maximum
value, $\Delta v=\textsf{t}$. This gives the following version of (\ref{umu}):
\begin{equation}\label{umuomega}
u\geq 0,\quad\mu\geq 0, \quad  \omega \cap {\rm supp}\,\mu =\emptyset.
\end{equation}
Here we have introduced the \emph{non-coincidence} set $\omega$ for the obstacle problem, defined as that open set in which there is definitely no contact between the solution $v$ and the obstacle $\psi$: 
\begin{equation}\label{noncoincidence1}
\omega =\{x\in M: \exists\, \varphi \in C^\infty(M) \text{ such that } 0\leq \varphi \leq u,\, \varphi (x)>0 \}.
\end{equation}
On the complementary set $M\setminus \omega$, the \emph{coincidence} set, the solution $v$ may exert a pressure on the obstacle, and this pressure is
represented by the measure $\mu$. 
It may happen that there is a nonempty set $M\setminus ( \omega \cup {\rm supp}\,\mu)$ left over, and on that set $v$ is in contact with the obstacle but exerts no pressure.
In \cite{Hedenmalm-Makarov-2013} points in that set are called ``shallow points''.

One may remark that (\ref{umuomega}) is a somewhat crude version of (\ref{umu}), but (\ref{umu}) itself need not make sense in the present generality. 
The original meaning of this integral is actually as a duality pairing $\left< \mu, u \right>$ between $\mu\in W^{-1,2}(M)_n$
and $u\in W^{1,2}(M)$, but one might also try to interpret it as a measure theoretic integral, which then could be written $\int u\,d\mu$.
However, $u$ is just the potential of a charge distribution, which means that it is locally the difference between two superharmonic functions. Such a function may
at a small set (of capacity zero) be of the form $(+\infty)-(+\infty)$, hence have no definite value there. And if $\mu$ loads such a point the integral will not be well-defined.

If $G^{\sigma_-}$ is assumed to be continuous the situation is better. Then $u$ is lower semicontinuous (see Remark~\ref{rem:Gcontinuous}), which means that the 
auxiliary function $\varphi$ in (\ref{noncoincidence1}) is not needed and $\omega$ can be defined directly as
\begin{equation}\label{noncoincidence}
\omega =\{x\in M: u (x)>0 \}.
\end{equation}
In this case the integral in (\ref{umu}) makes sense as a measure theoretic integral, and (\ref{umu}) and (\ref{umuomega}) then are equivalent.

The following simple lemma shows that the smaller $\sigma$ is, the bigger is the chance that  ${\rm Bal\,}(\sigma,0)$ exists.

\begin{lemma}\label{lem:sigmatilde}
Let $\sigma_j$ ($j=1,2$) be charge distributions such that $\sigma_1\leq \sigma_2$. Then, if  ${\rm Bal\,}(\sigma_2,0)$ exists so does ${\rm Bal\,}(\sigma_1,0)$.
\end{lemma}

\begin{proof}
Let $\textsf{t}_j$, $\psi_j$ be defined in terms of $\sigma_j$ as in Definition~\ref{def:partialbalayage2}. By assumption there exists a function $v_2$ satisfying $v_2\geq \psi_2$,
$\Delta v_2\leq \textsf{t}_2$. Set $v_1=v_2-\psi_2+\psi_1$. Then $v_1\geq \psi_1$ and
\begin{align*}
  d*dv_1 &= d*dv_2-(\sigma_2+\textsf{t}_2 {\rm vol}^n)+(\sigma_1+\textsf{t}_1 {\rm vol}^n) \\
         &= d*dv_2-\textsf{t}_2 {\rm vol}^n+\sigma_1-\sigma_2+ \textsf{t}_1 {\rm vol}^n\leq \textsf{t}_1 {\rm vol}^n.
\end{align*}
Thus ${\rm Bal\,}(\sigma_1,0)$ exists.
\end{proof}

The above result shows (among other things) that ${\rm Bal\,}(\sigma,0)$ may exist even if $\sigma_-$ has infinite energy: starting from any $\sigma$ 
($=\sigma_2$ in the lemma) for which
${\rm Bal\,}(\sigma,0)$ exists one may subtract any positive $n$-form current, even one with infinite energy, and the partial balayage still exists.
On the other hand, it is easy to give examples (satisfying (\ref{sigmastrictlynegative})) for which ${\rm Bal\,}(\sigma,0)$ does not exist. If for instance, in
dimension $n\geq 2$, $\sigma_-$ just consists of point masses, then  ${\rm Bal\,}(\sigma,0)$ will not exist, because even though $\sigma_+$ fits
into $\sigma_-$ in principle, it becomes too expensive to move it there (the cost in terms of energy would be infinite).

Illustrations will be given in Section~\ref{sec:examples}. The following result may be close to sharp.
\begin{theorem}\label{thm:sigmatau}
Assume that there exists a charge distribution $\tau$ of finite energy such that
$$
0\leq \tau\leq \sigma_- \quad \text{and}\quad \int_M \sigma_+\leq \int_M \tau.
$$
Then ${\rm Bal\,}(\sigma,0)$ exists.
\end{theorem}

\begin{proof}
By Lemma~\ref{lem:sigmatilde} it is enough to show that ${\rm Bal\,}(\sigma_+-\tau,0)$ exists. And that follows from Theorem~\ref{thm:partialbalayage1}.
\end{proof}


\section{Structure of partial balayage}\label{sec:structure}

In the sequel we shall always work with Definition~\ref{def:partialbalayage2} of partial balayage.
The natural bounds for it are given by

\begin{lemma}\label{lem:bounds} Whenever ${\rm Bal}(\sigma,0)$ exists it is subject to the bounds
\begin{equation}\label{bounds}
-\sigma_-\leq {\rm Bal\,}(\sigma,0)\leq 0.
\end{equation}
More generally,
\begin{equation}\label{boundsgeneral}
\min(\sigma, \lambda)  \leq {\rm Bal\,}(\sigma,\lambda)\leq \lambda.
\end{equation}
\end{lemma}

\begin{proof}
In view of (\ref{balgeneral}) the two bounds (\ref{bounds}) and (\ref{boundsgeneral}) are equivalent, so we need only discuss the first one. 
The upper bound holds by definition, but the lower bound, which can be written as
\begin{equation}\label{balsigmasigma}
{\rm Bal\,}(\sigma_+,\sigma_-)\geq 0,
\end{equation}
is not completely trivial. However, it turns out that the proof used in \cite{Gardiner-Sjodin-2009}
(see Theorem~4 there) for the Euclidean case carries over with minor changes. 

A sketch of the proof goes as follows.
The Green's potential of ${\rm Bal\,}(\sigma_+,\sigma_-)$ above is $V= G^{\sigma-}-v+C$, where $C$ is the constant in (\ref{uvG}), and it satisfies
$V\leq G^{\sigma_+}$.
First one proves the statement (\ref{balsigmasigma}) under the assumption that $G^{\sigma_-}$
is continuous. This entails that also  $V$  is continuous (see Remark~\ref{rem:Gcontinuous} above). 
The assertion to be proved  amounts to showing that $\Delta V\leq s$, where $s=\int_M\sigma_+$, and that can be verified by
checking a corresponding mean-value property. See  \cite{Gardiner-Sjodin-2009}, or \cite{Gustafsson-Sakai-1994}
(proof of Theorem~2.1), for details.

In case $G^{\sigma_-}$ is not continuous one approximates $V$ from below
by an increasing sequence of potentials (superharmonic minus a smooth compensating
term) and applies the previous argument to each of these. The details are given in  \cite{Gardiner-Sjodin-2009}.
\end{proof}

The more detailed structure of partial balayage says, roughly, that
only the two extremal values in (\ref{bounds}) or
(\ref{boundsgeneral}) are really attained.  Various ways of
formulating such a result, in the terminology of either the obstacle
problem or some form of balayage, appear in
\cite{Kinderlehrer-Stampacchia-1980, Friedman-1982,
  Gustafsson-Sakai-1994, Saff-Totik-1997, Sjodin-2007,
  Gardiner-Sjodin-2009, Hedenmalm-Makarov-2013}, to mention just a few
sources.  Working in the general setting of (\ref{boundsgeneral}) we
can always write
\begin{equation}\label{generalstructure}
{\rm Bal\,}(\sigma,\lambda)=\lambda|_\Omega + \sigma|_{M\setminus\Omega}+\nu, 
\end{equation}
where $\Omega$ is defined to be the largest open set in which ${\rm Bal\,}(\sigma,\lambda)=\lambda$, \ie 
$$
\Omega=M\setminus{\rm supp\,}(\lambda-{\rm Bal\,}(\sigma, \lambda)).
$$ 
This makes the first two terms in the right member of (\ref{generalstructure}) well-defined, so the equation as a whole can be viewed
simply as the definition of the unspecified term $\nu$ (for which ${\rm supp\,}\nu\subset M\setminus \Omega$ by definition of $\Omega$). 

Recall now the expression for partial balayage in terms of the potential $u$: 
$$
{\rm Bal\,}(\sigma,\lambda)=d*du +\sigma
$$
Here we remark that $u$, as well as $\omega$, $\Omega$ and $\nu$, remains unchanged under transformations as in
(\ref{covariance}). So there is no ambiguity when speaking about   $u$, $\omega$, $\Omega$, $\nu$ when changing
between ${\rm Bal\,}(\sigma,\lambda)$ and ${\rm Bal\,}(\sigma-\lambda,0)$, for example.
From the complementarity (\ref{complementaritysystem}) or (\ref{umuomega}) (where ${\rm supp\,}\mu=M\setminus \Omega$) 
we see immediately that  $\omega\subset\Omega$,  $\omega$ being the noncoincidence set (\ref{noncoincidence1}).
Hence $u=0$ in the open set $M\setminus (\Omega\cup\partial\Omega)$, so
\begin{equation}\label{suppnu}
{\rm supp\,}\nu \subset \partial\Omega.
\end{equation} 

Now we want to make this more precise, and eventually prove that $\nu=0$ under mild conditions.
Because of the gauge freedom (\ref{covariance}) we may assume that $\sigma,\lambda\geq 0$. This simplifies
reference to other work, in particular \cite{Sjodin-2007} and \cite{Gardiner-Sjodin-2009}, which will be crucial.

A first assumption needed is that $G^\lambda$ is a continuous function. Under this condition  (\ref{noncoincidence1})
takes the simpler form  (\ref{noncoincidence}). In addition,  (\ref{suppnu}) can be sharpened
to saying that $\nu$ lives on a subset of $\partial\Omega$ having ${\rm vol}^n$ measure zero, in other words,  $\nu$ is
singular with respect to ${\rm vol}^n$. In fact, when $G^\lambda$ is continuous, partial balayage can be
connected to the reduction operation in classical potential theory, as shown in  \cite{Gardiner-Sjodin-2009}
(see Theorem~7 there). And by using either a direct argument (as in Theorem~10 in \cite{Gardiner-Sjodin-2009}) or by referring to known results
\cite{Oksendal-1972, Bourgain-1987, Hansen-Hueber-1988} saying that harmonic measure
(defined in terms of reduction) lives on sets of Lebesgue measure zero, it follows that $\nu$ is singular with respect to ${\rm vol}^n$.

 \begin{remark}\label{rem:reduction}
The reduction operation in potential theory is, like partial balayage, defined in terms of an obstacle problem.
However, that obstacle problem goes in the opposite direction (compared to that for partial balayage), and for this reason
some minor assumptions (like continuity of $G^\lambda$ above) are needed to connect the two theories. \blsquarehere
\end{remark}

Next,  in \cite{Gardiner-Sjodin-2009} (again Theorem~7 there) the authors obtain bounds for $\nu$:
\begin{equation}\label{nulambda}
0\leq \nu\leq \lambda.
\end{equation}
Here the upper bound is actually a direct consequence of (\ref{boundsgeneral}), but the lower bound is not that easy to prove, despite it looks very natural
(it amounts to saying that $d*du\geq 0$ on the set where $u=0$; cf. \cite{Brezis-Ponce-2004}).
Now, if we make the additional assumption that  $\lambda$ is absolutely continuous with respect to ${\rm vol}^n$, then (\ref{nulambda}) together with
$\nu$ being singular with respect to ${\rm vol}^n$ forces $\nu$ to be zero. Hence we have the following theorem, which essentially is a restatement
of Theorem~10 in \cite{Gardiner-Sjodin-2009}.

\begin{theorem}\label{thm:structure} 
If $\sigma,\lambda\geq 0$, $G^{\lambda}$ is continuous and $\lambda$
is absolutely continuous with respect to ${\rm vol}^n$,
then the structure formula 
\begin{equation}\label{purestructure}
{\rm Bal\,}(\sigma,\lambda)=\lambda|_\Omega + \sigma|_{M\setminus\Omega}, 
\end{equation}
holds.

In particular, with $\sigma$  a general charge distribution, if $G^{\sigma_-}$ is continuous and $\sigma_-$
is absolutely continuous with respect to ${\rm vol}^n$, then the measure
\begin{equation}\label{muBal0}
\mu=-{\rm Bal\,}(\sigma,0),
\end{equation} 
has the simple structure
\begin{equation}\label{structure}
\mu=\sigma_-|_{{\rm supp\,}\mu}.
\end{equation}

\end{theorem}

In Section~\ref{sec:onedimension} we will give an example in one dimension showing that $\nu$ need not vanish
if (in the setting of (\ref{muBal0})) ${\sigma_-}$ has point masses, even if $G^{\sigma_-}$ is continuous. So at least the assumption that $\sigma_-$
is absolutely continuous is really necessary.

\begin{remark}
A  slightly more general version the structure formula is
\begin{equation}\label{structure1}
{\rm Bal\,}(\sigma, \lambda)=\lambda|_D + \sigma|_{M\setminus D},
\end{equation}
where $D$ is any set in the interval
$\omega\subset D\subset \Omega$.
In fact, S.~Gardiner, T.~Sj\"odin \cite{Sjodin-2007, Gardiner-Sjodin-2009}, as well as many other authors, work with the choice $D=\omega$. \blsquarehere
\end{remark}

\begin{remark}
Here we mention two ways of replacing the two assumptions on $\sigma$ in the second part of Theorem~\ref{thm:structure} by one single assumption.

First, in terms of the function $u$,   (\ref{structure}) says that $d*d u=0$ (as a measure) on ${{\rm supp\,}\mu}$, so to prove (\ref{structure}) it is by
(\ref{umuomega}) enough to prove that $d*d u=0$ on the coincidence set $M\setminus \omega$. 
If $\sigma$ is absolutely continuous with respect to ${\rm vol}^n$ with a density 
function in $L^p(M)$ for some $p>1$, then $u\in W^{2,p}(M)$ by the regularity theory for variational inequalities \cite{Kinderlehrer-Stampacchia-1980, Friedman-1982, Petrosyan-Shahgholian-Uraltseva-2012} (this regularity can also be derived from (\ref{bounds})).  Then the second derivatives of $u$ are in $L^p(M)$, and they vanish a.e. on 
the set where $u=0$ (see again \cite{Kinderlehrer-Stampacchia-1980}, Appendix A to Ch.~II). So (\ref{structure}) follows if $\sigma\in L^p(M)$, $p>1$
(cf. also Theorem~4.10 in \cite{Hedenmalm-Makarov-2013}). 

Alternatively, to make an assumption only on $\sigma_-$,  we may  assume that
$\sigma_-$ is absolutely continuous with respect to ${\rm vol}^n$ with a density function which belongs to $L^p(M)$ for some $p>n/2$.
Then $G^{\sigma_-}$ is in the Sobolev space $W^{2,p}(M)$, hence is continuous by the Sobolev  embedding theorem. This makes
both assumptions in Theorem~\ref{thm:structure} fulfilled. \blsquarehere
\end{remark}

\begin{remark}\label{rem:properties}
For the sake of completeness we mention a couple of further properties of partial balayage. First, the balayage operation
can always be broken up in smaller steps, in the precise sense that if $\lambda_1\leq \lambda_2 +\sigma_2$ then
$$
{\rm Bal}({\rm Bal}(\sigma_1,\lambda_2)+\sigma_2, \lambda_1)={\rm Bal}(\sigma_1+\sigma_2,\lambda_1).
$$
Combining this with the estimates (\ref{boundsgeneral}) easily gives the monotonicity 
$$
\sigma_1\leq \sigma_2 \Longrightarrow {\rm Bal}(\sigma_1,\lambda)\leq {\rm Bal}(\sigma_2,\lambda). 
$$
See  \cite{Gustafsson-Sakai-1994} for proofs and some more properties. \blsquarehere
\end{remark}


\section{Weighted equilibrium distributions}
\label{sec:weighted-eq-confs}

In \cite{Saff-Totik-1997} the theory of weighted equilibrium measures, \ie measures minimizing a certain energy functional under
the influence of some external field, is developed in the setting of two-dimensional logarithmic potential theory. The energy
functional used is
\begin{equation}
  \label{eq:equilib1}
  I_Q(\mu) = \int_{\mathbb{C}} U^\mu \, d \mu + 2 \int_{\mathbb{C}} Q \, d \mu,
\end{equation}
where $\mu$ is a Borel probability measure on $\C$, $U^\mu$ is the (Newtonian) potential of $\mu$, and $Q$ is a function on $\C$,
thought of as an external field. One interpretation of \eqref{eq:equilib1} is that the first term is the self-energy of the
measure $\mu$, and the second term the interaction energy of $\mu$ with the field $Q$. Under suitable assumptions on $Q$ it is
known that there exists a unique probability measure $\mu_Q$, the \emph{weighted equilibrium measure}, that minimizes $I_Q$ over
the set of all Borel probability measures.

In this section we utilize the a complementarity relationship developed in \cite{Roos-2015} between weighted equilibrium measures
and measures obtained from partial balayage operations to define a similar notion of weighted equilibrium $n$-forms on compact
Riemannian manifolds.

Let $M$ be an $n$-dimensional compact Riemannian manifold, and let $Q$ be a potential on $M$, assumed bounded from below.
We are going to treat the potential $Q$ as an external field that is applied on the manifold, and as such there will in a natural
way arise an $n$-form with similar properties as the weighted equilibrium measure in the complex setting. Let $\mathsf{t} > 0$ be
arbitrary but fixed---this will become a parameter that in essence tunes the total mass of the resulting $n$-form---and define
\begin{align}
  \label{eq:wec-sigma-t-def}
  \sigma_{\mathsf{t}} := - d * d Q - \mathsf{t} \, \mathrm{vol}^n.
\end{align}
From Example~\ref{ex:psiuppersemicontinuous} we know that ${\rm Bal}(\sigma_{\mathsf{t}}, 0)$ exists as $Q$ is assumed to be
bounded from below and since we have
\begin{align}
  Q = G^{\sigma_{\mathsf{t}} + \mathsf{t} \, \mathrm{vol}^n} + c = G^{\sigma_{\mathsf{t}}} + c
\end{align}
for some constant $c$. We now simply define the \emph{weighted
  equilibrium $n$-form} $\mu_{Q,\mathsf{t}}$ to be
\begin{equation}
  \label{eq:equilib3}
  \mu_{Q,\mathsf{t}} := -{\rm Bal}(\sigma_{\mathsf{t}}, 0).
\end{equation}
As a justification for this definition, we saw in the end of Section~\ref{sec:variational} that in the context of energy
minimization the calculation of ${\rm Bal(\sigma_{\mathsf{t}}, 0)}$ essentially boils down to finding the $n$-form $\nu$ that
minimizes the difference $\mathcal{E}(\nu) - 2\mathcal{E}(\nu, \sigma_{\mathsf{t}})$ over the set of $\nu$ satisfying
$\nu \leq 0$ and $\int_M \nu = \int_M \sigma_{\mathsf{t}} = -\mathsf{t}$ ($< 0$). This is equivalent to minimizing
$\mathcal{E}(-\mu) - 2\mathcal{E}(-\mu, \sigma_{\mathsf{t}})$ over the set of $n$-forms $\mu$ satisfying $\mu \geq 0$ and
$\int_M \mu = - \int_M \sigma_{\mathsf{t}} = \mathsf{t}$ ($> 0$), and an expansion of the energy difference shows that
\begin{align*}
  \mathcal{E}(-\mu) - 2\mathcal{E}(-\mu, \sigma_{\mathsf{t}}) &= \int_M G^{-\mu} \wedge (-\mu) - 2 \int_M G^{\sigma_{\mathsf{t}}}
                                                                \wedge (-\mu) \\
                                                              &= \int_M G^{\mu} \wedge \mu + 2 \int_M G^{\sigma_{\mathsf{t}}} \wedge \mu.
\end{align*}
Since $Q = G^{\sigma_{\mathsf{t}}} + c$ this is, up to a constant that does not matter in the minimization problem anyway, a
clear analogue to \eqref{eq:equilib1}. Moreover, definition \eqref{eq:equilib3} implies the following result, highly related to an
important characterization result for weighted equilibrium measures in the logarithmic setting \cite[Theorem
I.3.3]{Saff-Totik-1997}.
\begin{proposition}
  Let $Q$ be a potential on $M$, bounded from below, let $\mathsf{t} > 0$ be arbitrary and let $\mu_{Q,\mathsf{t}}$ be the
  resulting weighted equilibrium $n$-form as in \eqref{eq:equilib3}. Then there exists a constant $c_{\mathrm{Robin}}$, the \emph{modified Robin
    constant}, such that
  \begin{align}
    Q + G^{\mu_{Q,\mathsf{t}}} &\geq c_{\mathrm{Robin}} \text{ everywhere}, \label{eq:equilib-Robin-ineq} \\
    Q + G^{\mu_{Q,\mathsf{t}}} &= c_{\mathrm{Robin}} \text{ on } \supp \mu_{Q,\mathsf{t}}. \label{eq:equilib-Robin-eq}
  \end{align}
\end{proposition}
\begin{proof}
  The proposed result essentially follows immediately from \eqref{umuomega} and \eqref{noncoincidence} as we here have 
  \begin{equation}
    \label{eq:equilib4}
    u = G^{\sigma_{\mathsf{t}} - {\rm Bal}(\sigma_{\mathsf{t}}, 0)} + c' = G^{\sigma_{\mathsf{t}}} + G^{\mu_{Q,\mathsf{t}}} + c'
    = Q + G^{\mu_{Q,\mathsf{t}}} - c_{\mathrm{Robin}}
  \end{equation}
  for some constants $c, c'$, where $c_{\mathrm{Robin}} := c - c'$. Inequality \eqref{eq:equilib-Robin-ineq} follows from the
  nonnegativity of $u$, and $\{u > 0\} \cap \supp \mu_Q = \emptyset$ yields \eqref{eq:equilib-Robin-eq}.
\end{proof}


\section{Quadrature domains for subharmonic functions}\label{sec:quadrature}

There is an equivalent description of partial balayage in terms of quadrature formulas for subharmonic functions.
Construction of  quadrature domains for subharmonic functions was actually one of the main incentives for developing a theory of partial balayage, see 
\cite{Sakai-1979, Sakai-1982, Sakai-1983, Gustafsson-1990b, Shapiro-1992, Gustafsson-Sakai-1994}.
The following theorem is a simple result in this respect,  adapted to the formalism of the present paper.

\begin{theorem}\label{thm:quadrature}
Let $\sigma$, $\nu$  be charge distributions in $M$, and assume that $G^{\sigma_-}$ is continuous and that (\ref{sigmastrictlynegative}) holds for $\sigma$.
Then  $\nu={\rm Bal}(\sigma,0)$  if and only if $-\sigma_-\leq \nu\leq 0$  and
\begin{equation}\label{quadrature}
\int_M \varphi \wedge ( \nu-\sigma)\geq 0 
\end{equation}
for every upper semicontinuous  potential $\varphi$ in $M$ which satisfies $d*d\varphi \geq 0$ in $M\setminus {\rm supp\,}\nu$.
\end{theorem}

\begin{remark}
The test functions $\varphi$ are assumed to be upper semicontinuous, hence to be bounded from above. We then agree that the canonical representative
(\ref{canonical}) shall be used. It follows that the integral in (\ref{quadrature}) has a definite meaning, with the value possibly being $+\infty$
($\sigma_+$ may have point masses and $\varphi$ may attain the value $-\infty$). 
In addition, it follows that $\varphi$ is subharmonic as a function in  $M\setminus {\rm supp\,}\nu$. 
The uses of partial integration in the proof below is justified by  the potential $u$ being bounded from below (in addition to $\varphi$ being bounded from above), 
as explained in the discussion after (\ref{varphi12}). \blsquarehere
\end{remark}

\begin{proof} 
Assume first that $\nu={\rm Bal}(\sigma,0)$. Using (\ref{bounds}) we then have $-\sigma_-\leq \nu=\sigma+ d*du\leq 0$, 
where $u\geq 0$ is lower semicontinuous (by Remark~\ref{rem:Gcontinuous})
and vanishes on ${\rm supp\,}\nu$ (see (\ref{umu}),  (\ref{umuomega})). Hence
\begin{align*}
  \int_M\varphi\wedge (\nu-\sigma) &= \int_M \varphi\wedge d*du = \int_M u\wedge d*d\varphi \\ 
  &=\int_{M\setminus {\rm supp\,}\nu} u\wedge d*d\varphi \geq 0
\end{align*}
for every $\varphi$ as in the statement of the theorem.

In the other direction, assuming $-\sigma_-\leq \nu\leq 0$  and that (\ref{quadrature}) holds we may first choose $\varphi =\pm 1$, by which (\ref{quadrature}) gives that
$\int_M \nu=\int_M \sigma$. Therefore there exists a potential $u$ such that $\nu-\sigma=d*du$. This $u$ is determined only up to
an additive constant, and since $\int_M \nu=\int_M \sigma<0$  we can adapt this constant  so  that  $\int_M u\wedge \nu=0$.
Moreover, $u$ can be chosen to be lower semicontinuous since $G^{\sigma_-}$ is assumed to be continuous.

Next, let $\tau\geq 0$ be a positive charge distribution in $M$. Then $\int_M(\tau +c\nu)=0$ for a suitable $c >0$, 
hence $\tau+c\nu=d*d\varphi$ for some potential $\varphi$. Because of the assumed lower bound on $\nu$ and the continuity of $G^{\sigma_-}$,
$\varphi$ can be chosen to be upper semicontinuous.
As $d*d\varphi=\tau\geq 0$ in $M\setminus {\rm supp\,}\nu$,  $\varphi$ is an allowed test function in (\ref{quadrature})  and it follows that
\begin{align*}
  0\leq \int_M\varphi\wedge(\nu-\sigma) &=\int_M \varphi\wedge d*du=\int_M u\wedge d*d\varphi \\
                                        &=\int_M u\wedge (\tau+c\nu)=\int_M u\wedge \tau.
\end{align*}
Since $\tau\geq 0$ was arbitrary it follows that $u\geq 0$ in $M$. Together with $\nu\leq 0$ and$\int_M u\wedge \nu=0$ 
this shows that  $\nu=\sigma+d*du={\rm Bal\,}(\sigma,0)$. 

\end{proof}

A typical application is obtained by choosing $\sigma=\mu -{\rm vol}^n$, where $\mu\geq 0$ is sufficiently concentrated to a small set,
for example is singular with respect to ${\rm vol}^n$, or satisfies $\mu \geq {\rm vol}^n$ on some open set and vanishes
outside that set. In these cases the structure formula (Theorem~\ref{thm:structure}) shows that
\begin{equation}\label{balsubharm}
{\rm Bal}(\mu,{\rm vol}^n)= \chi_\Omega\, {\rm vol}^n
\end{equation}
for some saturated open set  $\Omega\subset M$. And by Theorem~\ref{quadrature} this is equivalent to $\Omega$ being a \emph{subharmonic quadrature domain}
(or open set) for $\mu$, in the sense that $\mu=0$ outside $\Omega$ and
\begin{equation}\label{quadratureformula}
\int_\Omega \varphi \,{\rm vol}^n \geq \int \varphi \wedge\mu
\end{equation}
for all potentials $\varphi$ which are subharmonic in $\Omega$. 
See  \cite{Sakai-1982, Shapiro-1992, Gustafsson-Shapiro-2005} for further information.

Choosing $\mu$ of the form $\mu=t \delta_a +\chi_D \, {\rm vol}^n$ ($t>0$) one gets the weak formulation of a standard version of Laplacian growth, to be
discussed in the next section.


\section{Laplacian growth}\label{sec:Laplacian growth}

Laplacian growth refers to domain evolutions driven by gradients of harmonic domain functions.  The standard case, which may also
be named ``motion by harmonic measure'', is that the domain function is the Green's function of the domain with pole at a fixed
point and zero Dirichlet boundary data. Detailed information and many references for Laplacian growth can be found in
\cite{Gustafsson-Teodorescu-Vasiliev-2014}.  The original connection between Laplacian growth, in the context of a fluid
dynamical interpretation of it in terms of Hele-Shaw flow, and quadrature domains (or more exactly moment preservation) was made
by S.~Richardson \cite{Richardson-1972}.  Laplacian growth on manifolds has previously been discussed in
\cite{Varchenko-Etingof-1992, Entov-Etingof-1997, Hedenmalm-Shimorin-2002, Hedenmalm-Olofsson-2005, Crowdy-2005}, for example.

To make everything precise in the above standard case, let for any subdomain $D\subset M$ with nontrivial complement (say
with ${\rm vol}^n(M\setminus D)>0$), and any $a\in D$, $g_{D}(\cdot,a)$ be the Dirichlet Green's function of $D$, determined by
\begin{gather*}
  \left\{
    \begin{array}{l l}
      -d * d g_D(\cdot, a) = \delta_a & \text{ in } D, \\[.1cm]
      \phantom{-d*d} g_D(\cdot, a) = 0 & \text{ on } \partial D.
    \end{array}
  \right.
\end{gather*}
Then the dynamical law for the corresponding  time evolution $t\mapsto D(t)$, {\it Laplacian growth}, can be expressed as
\begin{equation}\label{LGlaw}
\frac{d}{dt}\int_{D(t)}\varphi \,{\rm vol}^n
=-\int_{\partial D(t)} \varphi * dg_{D(t)}(\cdot, a),
\end{equation}
which is to hold for every smooth test function $\varphi$ in $M$. The law says that the velocity of the boundary
$\partial D(t)$ in the outward normal direction equals minus the outward normal derivative of the Green's function.
Otherwise said, the velocity vector ${\bf v}(\cdot,t)$ by which the boundary moves is  
\begin{equation}\label{vgradg}
{\bf v}(x,t)= -\nabla g_{D(t)}(x,a),
\end{equation}
for $x\in \partial D(t)$. 

The formula (\ref{vgradg}) can be alternatively expressed, in the language of differential forms,  as
\begin{equation}\label{ivdg}
i({\bf v}(x,t)){\rm vol}^n=-*dg_{D(t)}(x, a),
\end{equation}
where $i(\cdot)$ denotes interior derivation (see \cite{Frankel-2012}). Recall also that   (minus) the $(n-1)$-form $*dg_D(\cdot,a)$ represents the harmonic
measure on $\partial D$ with respect to $a$, or the result of classical balayage of $\delta_a$ to $\partial D$.
It should be emphasised that the test function $\varphi$ in (\ref{LGlaw}) is to be independent of $t$. The relationship between (\ref{LGlaw}) and (\ref{ivdg}) then becomes immediate
from H.~Cartan's formula for the Lie derivative $\mathcal{L}_{\bf v}$ acting on forms, combined with Stokes' theorem: 
\begin{align*}
  \frac{d}{dt}\int_{D(t)} &\varphi \,{\rm vol}^n = \int_{D(t)} \mathcal{L}_{{\bf v}} (\varphi \,{\rm vol}^n) \\
  =&\int_{D(t)} (d ( i({{\bf v}}) (\varphi \,{\rm vol}^n)+ i({{\bf v}}) d(\varphi \,{\rm vol}^n)) =\int_{\partial D(t)} \varphi\, i({{\bf v}})\, {\rm vol}^n.
\end{align*}

For test functions $\varphi$ which are subharmonic in $D(t)$ one has 
$$
-\int_{\partial D(t)} \varphi * dg_{D(t)}(\cdot, a)\geq \varphi (a).
$$
Hence on using only such functions and integrating  (\ref{LGlaw}) from time zero to 
some positive time $t$ one gets
$$
\int_{D(t)} \varphi\, {\rm vol}^n -\int_{D(0)} \varphi \,{\rm vol}^n\geq t \varphi(a).
$$
This inequality, holding for test functions $\varphi$ which are subharmonic in $D(t)$, represents
a weak formulation of the Laplacian growth law. It says that $D(t)$ is a subharmonic quadrature domain
for the measure $t\delta_a + \chi_{D(0)}{\rm vol}^n$,  and by Theorem~\ref{thm:quadrature} it is equivalent
to the balayage statement
$$
{\rm Bal}(t\delta_a+\chi_{D(0)}\, {\rm vol}^n, {\rm vol}^n)= \chi_{D(t)} \, {\rm vol}^n.
$$

In the above weak formulations one may  start with an arbitrary initial open set $D(0)\subset M$ and allow any
$0<t<{\rm vol}^n(M\setminus D(0))$. In addition, the point $a$ need not be in $D(0)$, in fact
$D(0)$ may even be the empty set. This gives rise to what will be called harmonic balls in the next section.


\section{Harmonic and geodesic balls}\label{sec:balls}

There are two kinds of balls to consider,  geodesic balls and  harmonic balls. The geodesic balls are 
simply the ordinary balls defined in terms of the Riemannian distance function, while the harmonic balls are defined 
by partial balayage, or by mean-value properties for harmonic functions. In the Euclidean case these two kinds of balls coincide. 
Here we shall prove that, in two dimensions, geodesic and harmonic balls agree if and only if the Gaussian curvature of the
manifold is constant. We begin with the definitions.

\begin{definition}\label{def:balls}
Let $M$ be a Riemannian manifold.
The {\it geodesic ball} with centre $a\in M$ and radius $r>0$ is
$$
B^{\rm geod} (a,r)=\{x\in M: {\rm dist}(x,a)<r\},
$$
where ${\rm dist}(x,a)$ denotes geodesic distance between $x$ and $a$.

The {\it harmonic ball} with centre $a\in M$ and volume $t>0$ is the open saturated
set $B=B^{\rm harm} (a,t)$ defined by
\begin{equation}\label{balB}
{\rm Bal\,}(t\delta_a, {\rm vol}^n)=\chi_{B} \, {\rm vol}^n.
\end{equation}
 \end{definition}

In the above definition, the balayage statement (\ref{balB}) can be replaced by the quadrature property (see Section~\ref{sec:quadrature}) that 
$$
\int_{B}\,h \,{\rm vol}^n \geq t h(a)
$$ 
holds for all integrable subharmonic functions $h$ in $B$. 
At least for small values of $t$ (and we shall consider only such values),
it is known that the \emph{a priori} weaker mean-value property that
\begin{equation}\label{harmmean}
\int_{B}\,h \,{\rm vol}^n = t h(a)
\end{equation} 
holds for all integrable harmonic functions $h$ in $B$ is enough to ensure (\ref{balB}).

In general, the study of mean-value properties such as (\ref{harmmean}) has a long history,
which can be traced back even to I.~Newton. In fact, Newton proved that the exterior gravitational field of a homogeneous ball
is the same as that of a point mass in the centre, and that statement is equivalent to (\ref{harmmean}).
Some general discussion of mean-value properties, as well as further references, can be found in \cite{Shapiro-1992}. 
The specific notion of harmonic ball was introduced in \cite{Shahgholian-Sjodin-2013} in
the Euclidean case, and studies for curved manifolds, in the context of the corresponding Hele-Shaw flow
problem (Laplacian growth), can be found in \cite{Varchenko-Etingof-1992, Hedenmalm-Shimorin-2002}, for example.

In the (locally) Euclidean case, \ie  with $ds^2=dx_1^2+\dots +dx_n^2$, we clearly
have $B^{\rm geod} (a,r)=B^{\rm harm} (a,t)$ with $r$ and $t$ related by
\begin{equation}\label{trGamma}
t=r^n\,\frac{\pi^{n/2}}{\Gamma (n/2 + 1)}.
\end{equation}
Here the last factor is the volume of the unit ball in $n$ dimensions.
It is easy to see, however, that geodesic and harmonic balls cannot always be the same in the case curved manifolds.
The main result in the present section is the following.

\begin{theorem}\label{thm:equivalentballs}
Let $M$ be a Riemannian manifold of dimension two. Then (small) geodesic
and harmonic balls with the same centre coincide, as families, if and only if the Gaussian curvature $\kappa$
of $M$ is constant.
When this holds, then the relationship between the balls is more precisely 
\begin{equation}\label{geodharm}
B^{\rm geod} (a,r)=B^{\rm harm} (a,t)
\end{equation}
with $r,t>0$ related by 
\begin{equation}\label{trGammakappa}
t=
\begin{cases}
\displaystyle \phantom{-} \frac{\pi}{\kappa} \sin^2(\sqrt{\kappa}r) \quad &\text{if } \kappa >0,\\[.2cm]
\phantom{-} \pi r^2 \quad &\text{if } \kappa =0,\\[.2cm]
\displaystyle -\frac{\pi}{\kappa} \sinh^2(\sqrt{-\kappa}r) \quad &\text{if } \kappa <0.\\
\end{cases}
\end{equation}
\end{theorem}

\begin{remark}
The theorem is local in nature, and we do not require $M$ to be compact. It may be
just a small subdomain of a compact manifold, for example. \blsquarehere
\end{remark}

The natural framework for dealing with geodesic balls in two dimensions is geodesic polar coordinates, 
and   we start by giving a short discussion of such coordinates. 
Some more details can be found in  \cite[Section~10.3]{Frankel-2012} and \cite[Section~5.6]{Petersen-2006}.

Geodesic polar coordinates $(r,\varphi)$ centred at a point $a\in M$ bring the metric to the form
\begin{equation}\label{geodesiccoordinates}
ds^2=dr^2+\rho (r,\varphi)^2\, d\varphi^2
\end{equation}
for some function $\rho(r,\varphi)>0$. Here $\varphi$ is an angular parameter with period $2\pi$ and $r$ equals
the geodesic distance from the coordinate origin (\ie the point $a\in M$) to the point with coordinates $(r,\varphi)$. 
Geodesic polar coordinates exist only in a neighbourhood of $a$, so small that that geodesic balls are still topological balls,
more precisely up to the injectivity radius, the largest radius for which the exponential map is a diffeomorphism
(see  \cite[Section~5.9.2]{Petersen-2006} and \cite[Section~6.5]{Berger-2003}). 
Since there are no  mixed terms in (\ref{geodesiccoordinates}) the coordinates are orthogonal. The
point $a$ itself is singular for the geodesic coordinates. To account for this singularity, the function $\rho$ has to satisfy,
in the limit $r \to 0$,
\begin{equation}\label{limitsrho}
  \left\{
    \begin{array}{r}
      \phantom{\partial}\rho=0,\\[.2cm]
      \displaystyle \frac{\partial \rho}{\partial r}=1.
    \end{array}
  \right.
\end{equation}
More precisely, Taylor expansion of $\rho$ with respect to $r$ gives 
\begin{equation}\label{limitsrhoprecise}
\rho(r,\varphi)=r +r^2\sigma(r,\varphi)), 
\end{equation}
for some smooth function $\sigma(r,\varphi)$, $2\pi$-periodic in $\varphi$.

The level lines of $\varphi$ are exactly
the geodesic curves emanating from $a\in M$, and the level lines of $r$ are the 
geodesic spheres (the boundaries of the geodesic balls) centred at $a$. 
Thus the geodesic ball with radius $R>0$ is given in geodesic polar coordinates by
$$
B^{\rm geod} (a,R)=\{(r,\varphi): 0<r<R\}\cup \{a\}.
$$
The Gaussian curvature $\kappa$ of the metric (\ref{geodesiccoordinates}) is obtained from
\begin{equation}\label{rhokapparho}
\frac{\partial^2 \rho}{\partial r^2}+\kappa\rho=0,
\end{equation}
see again \cite{Frankel-2012}.

\begin{example}\label{ex:geodesiccoordinates}
If $\kappa$ is independent of $r$
then (\ref{rhokapparho}),  (\ref{limitsrho}) can be immediately integrated to give
\begin{equation}\label{rhosin}
\rho(r,\varphi)=
\begin{cases}
\displaystyle \frac{1}{\sqrt{\kappa}}\sin(\sqrt{\kappa}r) &\text{if } \kappa >0,\\[.2cm]
r\quad &\text{if } \kappa =0,\\[.2cm]
\displaystyle \frac{1}{\sqrt{-\kappa}}\sinh(\sqrt{-\kappa}r)\quad &\text{if } \kappa <0.
\end{cases}
\end{equation}
For $\kappa=0$ we can identify (\ref{rhosin}) with the standard polar coordinates in the Euclidean plane.
As an example with $\kappa>0$ we may let $a$ be the north pole on the sphere $M=\partial B(0,R)$  in $\R^3$ ($R>0$ being a fixed radius). On that
sphere we have the ordinary spherical coordinates $(\theta,\varphi)$ with $0<\theta<\pi$ and $\varphi$ being $2\pi$-periodic.
Then the coordinate origin $(0,0)$ corresponds to the point $a$, and the metric on $M$ is
$$
ds^2= R^2 \, d\theta^2 + R^2 \sin^2\theta \, d\varphi^2.
$$
This is of the form (\ref{geodesiccoordinates}) with
$$
\begin{cases}
r = R\theta,\\[.2cm]
\varphi= \varphi \quad (\rm{unchanged}),\\[.2cm]
\displaystyle \rho(r,\varphi)= R \sin \frac{r}{R}.
\end{cases}
$$
With $\kappa =1/R^2$ we can identify this with the first option in (\ref{rhosin}).
Compare similar examples in \cite{Berger-2003}.
\end{example}

In the example above, neither $\rho$ nor $\kappa$ depends on $\varphi$.
This signifies that motions by the vector field $\partial/\partial \varphi$ are rigid 
transformations. The following lemma gives some equivalent statements in this
respect,

\begin{lemma}\label{lem:killing}
For arbitrary geodesic polar coordinates centred at a point $a\in M$, the following statements are equivalent.

\begin{itemize}

\item[(i)] 
$\displaystyle \frac{\partial}{\partial \varphi} \,\rho(r,\varphi)=0$.

\item[(ii)] 
$\displaystyle \frac{\partial }{\partial \varphi}\, \kappa(r,\varphi)=0$.

\item[(iii)] $\displaystyle \frac{\partial}{\partial \varphi}$ is a Killing vector field.

\item[(iv)] $\varphi$ is a harmonic function (locally). 

\end{itemize} 

\end{lemma}

Recall that the meaning of $(iii)$ is that the Lie derivative
by the vector field $\xi=\frac{\partial}{\partial \varphi}$ acting on the metric tensor vanishes:
$$
\mathcal{L}_{\xi} (dr\otimes dr +\rho(r,\varphi)^2 \, d\varphi \otimes d\varphi )=0,
$$
and that this can be interpreted as saying that the flow defined by $\xi$
is a one-dimensional flow by isometries.

\begin{proof} 
The equivalence  $(i)\Longleftrightarrow(ii)$  is obvious from the system (\ref{rhokapparho}),  (\ref{limitsrho}),
which is an initial value problem on standard form for $\rho$ as a function of $r$, having $\varphi$ as a parameter.

As for $(iii)$ we simply compute the Lie derivative in the local coordinates given. The result is
$$
\mathcal{L}_{\xi} (dr\otimes dr +\rho(r,\varphi)^2 \, d\varphi \otimes d\varphi )=\frac{\partial(\rho(r,\varphi)^2)}{\partial \varphi}\,d\varphi\otimes d\varphi,
$$ 
from which we immediately obtain $(i)\Longleftrightarrow(iii)$.

Turning to $(iv)$, a function $h$ is harmonic if and only if $d*dh=0$, and in geodesic polar coordinates this spells out to
$$
\frac{\partial}{\partial r} \left( \rho(r,\varphi) \, \frac{\partial h}{\partial r} \right)+ \frac{\partial}{\partial \varphi} \left( \frac{1}{\rho(r,\varphi)} \, \frac{\partial h}{\partial \varphi} \right) = 0.
$$
Thus $h=\varphi$ is harmonic if and only if $\frac{\partial}{\partial \varphi}(\frac{1}{\rho(r,\varphi)})=0$, \ie if and only if $(i)$ holds.

\end{proof}

When the equivalent conditions in Lemma~\ref{lem:killing} hold, then the conjugate harmonic function of $\varphi$ is, up to 
additive and multiplicative constants, the ordinary Green's function $g_R=g_R(\cdot,a)$ (with pole at $a$) 
for $B^{\rm geod}(a,R)$, for any small $R$. 
In fact, if the conjugate harmonic function of $\varphi$ is denoted $\psi$ the defining relationship is
$$
d\psi=*d\varphi.
$$
Now, using that $\rho$ does not depend on $\varphi$  (by assumption), 
$$
*d\varphi=\frac{1}{\rho(r)}(*\rho(r)d\varphi)=-\frac{1}{\rho(r)}\,dr,
$$
so 
\begin{equation}\label{psiint}
\psi=-\int \frac{dr}{\rho(r)}.
\end{equation}
This indefinite integral contains a constant of integration, which may depend on $\varphi$, and that constant can be adjusted so that $\psi = 0$ on $\partial B_R$. 
The strength of the singularity of $\psi$ at the point $a$ is linked to the increase of $\varphi$ by $2\pi$ as $a$ is encircled. 
Altogether we find that the function
$$
g_R = \frac{1}{2\pi} \psi,
$$
is exactly the (ordinary) Green's function for $B_R$, in the sense that it satisfies $-d*dg_R =\delta_a$ in $B_R$ and has boundary values $g_R=0$ on $\partial B_R$.

\begin{example}
If $\kappa$ is constant we can evaluate the integral (\ref{psiint}) by using the  explicit expressions (\ref{rhosin}) for $\rho$.
When $\kappa>0$, for instance, this gives
$$
\psi= \log \left| \cot \frac{\sqrt{\kappa}r}{2} \right| + {\rm constant},
$$ 
and so
$$
g_R(r)= \frac{1}{2\pi} \left( \log \left| \cot \frac{\sqrt{\kappa}r}{2} \right| - \log \left| \cot \frac{\sqrt{\kappa}R}{2} \right| \right) \quad (0<r<R).
$$ 
\end{example}

Now we turn to the proof of the theorem.

\begin{proof}[Proof of Theorem~\ref{thm:equivalentballs}]
Assume first  that $\kappa$ is constant and fix a point ${a\in M}$. 
Lemma~\ref{lem:killing} shows that, for geodesic polar coordinates $(r,\varphi)$ centred at $a$, $\xi=\frac{\partial}{\partial \varphi}$ is a Killing vector field. 
Thus the flow by $\xi$ consists of  rigid transformations which keep $a$ fixed. But it is obvious from construction that both types of balls,  geodesic
and harmonic balls centred at $a$, are uniquely determined by their radii or volumes and that they are invariant under such transformations. It follows
that  the two  families of balls are the same. The relationships (\ref{trGammakappa})
between $r$ and $t$ are obtained by elementary calculations.

For the other direction of the theorem, assume that (\ref{geodharm}) holds for all small $r,t>0$ and all $a\in M$. 
We start by fixing one point $a\in M$ and choosing geodesic polar coordinates $(r,\varphi)$ centred at $a$.
The mean-value property
$$
\int_{B_R}\, h \,{\rm vol}^2 = {\rm vol}^2(B_R)\,h(a)
$$ 
holds (by assumption (\ref{geodharm})) for $B_R=B^{\rm geod}(a,R)$, for every small $R>0$ and all integrable harmonic functions $h$ in $B_R$.
Differentiation of this identity with respect to $R$ gives the corresponding mean-value identity in terms of boundary integrals:
$$
\int_{\partial B_R}\, h \,{\rm vol}^1 = {\rm vol}^1(\partial B_R)\,h(a).
$$ 
This holds for harmonic function $h$ which are, say, continuous up to $\partial B_R$. The differentiation gives more precisely
the one-dimensional volume element in form of an interior derivative of the two-dimensional volume element, as 
$$
{\rm vol}^1=i \left( \frac{\partial}{\partial r} \right) {\rm vol}^2 = i \left( \frac{\partial}{\partial r} \right) \rho(r,\varphi) \, dr\wedge d\varphi=\rho(r,\varphi) \, d\varphi.
$$

On the other hand, we have quite generally a similar  identity with the boundary integral weighted with the normal derivative of the Green's function.
In the language of differential forms this looks
$$
-\int_{\partial B_R}\, h \, *dg_R = h(a).
$$   
On comparison we conclude that
\begin{equation}\label{alongdB}
-*dg_R = \frac{1}{ {\rm vol}^1(\partial B_R)}\, \rho(r,\varphi) \, d\varphi \quad {\rm along}\,\, \partial B_R.
\end{equation}
In general terms we have
\begin{align*}
  *dg_R &= * \left( \frac{\partial g_R}{\partial r}\,dr +\frac{1}{\rho(r,\varphi)} \frac{\partial g_R}{\partial \varphi}\,\rho(r,\varphi) \, d\varphi \right) \\
        &= -\frac{1}{\rho(r,\varphi)} \frac{\partial g_R}{\partial \varphi}\,dr +\frac{\partial g_R}{\partial r}\,\rho(r,\varphi) \, d\varphi,
\end{align*}
so (\ref{alongdB}) spells out to
\begin{equation}\label{gRboundary}
\begin{cases}
\displaystyle -\frac{\partial g_R}{\partial r}= \frac{1}{ {\rm vol}^1(\partial B_R)}\quad &{\rm on}\,\, \partial B_R,\\[.4cm]
\displaystyle \phantom{-}\frac{\partial g_R}{\partial \varphi}=0 \quad &{\rm on}\,\, \partial B_R.
\end{cases}
\end{equation}

Now $g_R$ is harmonic in $B_R\setminus \{a\}$ and has a fixed singularity at $a$.  
Thus, if we differentiate $g_R$ with respect to $R$ 
we obtain a harmonic function $\partial g_R/\partial R$ in $B_R$ without singularities.
On the boundary $\partial B_R$ we have
$$
g_R(R,\varphi)=0,
$$
and differentiating this with respect to $R$ gives
$$
\frac{\partial g_R}{\partial R}+\frac{\partial g_R}{\partial r}=0 \quad {\rm on}\,\,\partial B_R.
$$
In view of (\ref{gRboundary}) and the harmonicity of $\partial g_R/\partial R$ in $B_R$ this entails
$$
\frac{\partial g_R}{\partial R}= \frac{1}{ {\rm vol}^1(\partial B_R)} \quad {\rm in}\,\, B_R.
$$

Integrating the above identity from some $R_0<R$ to $R_1=R$ gives, for $0<r<R_0$,
$$
g_{R}(r,\varphi) -g_{R_0}(r,\varphi) =\int_{R_0}^R  \frac{dt}{ {\rm vol}^1(\partial B_t)}.
$$
On letting $r\to R_0$ the second term disappears and we get the explicit formula
$$
g_R(r,\varphi)=\int_{r}^R  \frac{dt}{ {\rm vol}^1(\partial B_t)}.
$$ 

In particular we see that $g_R$ does not depend on $\varphi$ and that the gradient of $g_R$
in addition does not depend on $R$:
$$
\nabla g_R = -\frac{1}{{\rm vol}^1(\partial B_r)}\,\frac{\partial }{\partial r}.
$$
A side remark here is that this says that the Hele-Shaw flow moving boundary problem, which has (minus)
the gradient of the Green's function as its velocity field in the fluid region, is a stationary flow in the
present situation, namely when the flow is driven a point source and starts from empty space in a constant curvature 
(two-dimensional) manifold. 

Knowing now that $g_R$  is independent of $\varphi$, the fact that $g_R(r,\varphi)=g_R(r)$ is a harmonic function (for $0<r<R$)
becomes
$$
\frac{\partial}{\partial r} \left( \rho(r,\varphi)\frac{\partial g_R(r)}{\partial r} \right) = 0.
$$
From this it follows  that
$\frac{\partial}{\partial r}\log \rho(r,\varphi)$ is a function only of $r$, so that
$$
\frac{\partial^2}{\partial \varphi \, \partial r}\log \rho(r,\varphi)=0.
$$   
But the general solution of the latter equation is of the form
$$
\log \rho({r,\varphi})= A(r)+B(\varphi),
$$
and using the behaviour (\ref{limitsrhoprecise}) of $\rho$ as $r\to 0$ one deduces that the function $B(\varphi)$
must be constant. Thus $\rho(r,\varphi)$ depends only on $r$.

Using Lemma~\ref{lem:killing} it now follows that also $\kappa$ depends only on $r$. But so far we have used the assumption (\ref{geodharm})
only at one point, $a\in M$. Repeating the same procedure for a nearby point gives the final
conclusion that $\kappa$ indeed is constant.
\end{proof}

\begin{remark}
It is possible to prove the easy direction the theorem (that constant curvature implies coincidence of balls)
by using a local version of the (rather deep) uniformisation theorem, namely by
introducing a local complex coordinate $z$ in which the metric takes the form 
\begin{equation}\label{dskappa}
ds=\frac{2|dz|}{1+\kappa |z|^2}.
\end{equation}
Taking then $z=0$ to correspond to the point $a\in M$ the proof becomes very easy.

A somewhat related observation is the following. Consider again the constant curvature metric (\ref{dskappa}), restricting to $|z|<1/\sqrt{-\kappa}$ if $\kappa$ is negative.
Then, by Theorem~\ref{thm:equivalentballs}, the class of geodesic balls coincides with the class of harmonic balls, and their centres agree. 
What we wish to remark here is that this class of balls in addition coincides with the usual Euclidean disks, 
but that the centres then will be different. In other words, any Euclidean disk $|z-a|<r$ in the complex plane is a geodesic and harmonic ball with respect to  (\ref{dskappa}),
but its centre as such a ball will depend on $\kappa$, and will in particular not coincide with $a$ unless $\kappa=0$ (or $a=0$). \blsquarehere

\end{remark}


\section{Remarks on noncompact manifolds}\label{sec:nonclosed}

We remark here on the modifications needed for the case of a manifold $M$ with boundary $\partial M$, which then itself is a 
manifold (of one dimension lower). For simplicity, we shall stay within the finite energy setting, and then the treatment can be based
on Hodge decompositions for manifolds with boundary, see \cite{Schwarz-1995}. 
We shall make no attempt of giving a complete theory of partial balayage on open manifolds in this paper.

Having a boundary means that boundary conditions have to be taken into account. 
Starting out from Definition~\ref{def:partialbalayage} there are  several natural options on what to impose on $u$:

\begin{itemize}

\item Dirichlet data: $u=0$ on $\partial M$.

\item Relaxed Dirichlet data: $u=\text{free constant}$ on $\partial M$, together with 
\begin{equation}\label{zeroflow}
\int_{\partial M}*du=0.
\end{equation}

\item Hydrodynamic type data: $du=0$ along $\partial M$, together with the zero flux condition (\ref{zeroflow}) holding for each individual component 
of $\partial M$. 
(If $\partial M$ has only one component this case is the same as the previous.)

\item Neumann data: $*du=0$ along $\partial M$.
\end{itemize}

The case of Dirichlet data is quite straight-forward:  $u$ is then to belong to the Sobolev space $W^{1,2}_0(M)$, where now the subscript $0$ signifies zero boundary values, 
and the theory becomes based on the isometric isomorphism
\begin{equation}\label{WW}
-d*d: W^{1,2}_0(M)\to W^{-1,2}(M)_n.
\end{equation}
Not all elements of  $W^{-1,2}(M)_n$ are charge distributions, but those which are make up a dense
subset of $W^{-1,2}(M)$ (see \cite{Treves-1975} for the Euclidean case). The Green's operator, or rather the map taking
charge distributions to Green's potentials, is simply the inverse of (\ref{WW}),
$$
G: W^{-1,2}(M)_n\to W^{1,2}_0(M).
$$
Green's potentials are defined accordingly, and everything works out with minor modifications (simplifications actually), compared to the compact case.

As for the balayage process, say in the form  $\sigma\mapsto {\rm Bal}(\sigma,0)$, it is important to take into account that some (or even all) 
of the mass may go to the boundary, and this mass should be kept track of,
even though, in the present paper, the notation ${\rm Bal}(\sigma,0)$ refers only to the mass within $M$.
The assumption (\ref{sigmanegative}) is not needed in the case of Dirichlet boundary conditions, and one may
even start with a positive current: $\sigma\geq 0$. In that case all mass will go to the boundary and partial balayage
will simply be the same as classical balayage of $\sigma$ to $\partial M$. The resulting measure  is absolutely continuous 
with respect to $(n-1)$-dimension measure on $\partial M$ and its density is (minus)  the outward normal derivative 
$\partial u/\partial n$ of the function $u$ in (\ref{complementaritysystem}). In other words, it is represented by 
(minus) the $(n-1)$-form $*du$ on $\partial M$. 

The other three types of boundary conditions all ensure mass balance within $M$:
$$
 \int_M{\rm Bal}(\sigma,0)-\int_M\sigma=\int_M d*du=\int_{\partial M}*du=0.
$$
Therefore  (\ref{sigmanegative}) will be a necessary assumption in these cases.
It should also be noted that all four kinds of boundary conditions guarantee partial integration free of boundary terms: if $u$ satisfies anyone of the mentioned
boundary conditions and $v$ satisfies the same, then
\begin{equation}\label{partialintegration}
\int_M du\wedge*dv= \int_{\partial M} u\wedge *dv -\int_M u\wedge d*dv= -\int_M u\wedge d*dv.
\end{equation}
This is important because it makes the theory for compact manifolds carry over smoothly, with only minor changes, to the case of manifolds with boundary.
For example, the potential $u$ in the definition of partial balayage will in all cases be characterized by the complementarity system (\ref{complementaritysystem}),
together with the given boundary conditions. 

The difference between vanishing Dirichlet data and those which are ``relaxed''  is essentially the requirement of mass conservation (\ref{zeroflow}) within $M$, which may be expressed as 
\begin{equation}\label{intdud}
\int_{M} d*du=0. 
\end{equation}
If one adds this condition directly to the zero Dirichlet data, then one can no longer infer that $u\geq 0$, and the complementarity system fails in the way it is written
in (\ref{complementaritysystem}). However, one can recover these properties if one just adjusts the additive level of $u$ in the same way as was done in the beginning of Section~\ref{sec:variational},
so that (\ref{complementarity}) holds. But then $u$ will (in general) not be zero on the boundary, it may take another constant value, which cannot be prescribed in advance.

So this is the meaning of the relaxed Dirichlet data. It is closely related to boundary conditions which are used for the stream function in two dimensional fluid mechanics.
This stream function takes free constant values on each boundary component, expressing that the flow is parallel to the boundary. This is then complemented by prescribing
the circulations around the holes, in accordance with Kelvin's theorem (conservation of circulations).

Thus hydrodynamic data requires that $du=0$ along $\partial M$, and that $*du$ is exact in a neighbourhood of $\partial M$. The case of Neumann data is more or less a dual version
of this: $*du=0$ along $\partial M$, while $du$ is exact already from outset.
There is a rather elegant way of reducing boundary value problems with zero Neumann data to the case of a compact manifold by means of 
a doubling procedure of P.E.~Conner \cite{Conner-1954, Conner-1956} and K.O.~Friedrichs \cite{Friedrichs-1955}. In the case of two dimensions, the corresponding idea
goes back to F.~Schottky \cite{Schottky-1877}.
The doubling procedure means more precisely that one adds, to ${M\cup \partial M}$, a copy $\tilde{M}$ of $M$ and glues it along the boundary
so that a compact manifold $\hat{M}= M\cup \partial M\cup \tilde{M}$ is obtained. 
The differentiable structure of $\hat{M}$ requires that the gluing is made via coordinate maps which (locally) take $\partial M$ into ${\{x_n=0\}\subset \R^n}$ and
neighbouring parts of $M$ into $\{x_n>0\}$; such maps are postulated in the definition of a manifold with boundary (see \cite{Schwarz-1995}). Then $\tilde{M}$ has the
corresponding maps, and before gluing one changes the orientation of $\tilde{M}$ by composing its coordinate maps with the reflection  $x_n\mapsto -x_n$. 
Eventually one pastes along $x_n=0$ in coordinate space. The resulting compact manifold (with its metric) will in general not be smooth across the boundary, 
but for appropriate choices of coordinates one can ensure that the metric tensor becomes Lipschitz continuous (see again \cite{Conner-1956, Friedrichs-1955}), 
which is good enough for the idea to work. An example, and some further discussion will be provided in Section~\ref{sec:double of ball}.

As for the partial balayage in the Neumann case, one turns it into a problem in $\hat{M}$ by taking the same data on $\tilde{M}$
as on $M$. Then, on $\hat{M}$, one has data which are symmetric with respect to the natural involution on $\hat{M}$. 
This enforces homogeneous Neumann data on $\partial M$.

For open manifolds in general, a natural method is to try exhaust the manifold by manifolds with boundary and select suitable boundary conditions for these.
It turns out that Dirichlet boundary conditions for the exhausting sequence is not a good choice because some mass is moved to the boundary, and this mass may 
eventually be lost in the limit. In Section~\ref{sec:ball} we give an example showing that this can occur in dimension $n\geq 3$, with $M=\R^n$.
On the other hand, we also show (Theorem~\ref{thm:excess}) that this does not happen in dimensions $n=1,2$.

The Euclidean case $M=\R^n$ can also be treated directly. It is natural to insist on mass conservation, \ie that  (\ref{intdud}) holds.
As indicated by (\ref{uvG}) one  expects the function $u$ to behave at infinity as a Newtonian or logarithmic ($n=2$) potential of a
compactly supported zero net mass distribution, modulo an additive constant. This means that
\begin{equation}\label{asymptoticsu}
du(x)=\mathcal{O} \left( \frac{1}{|x|^{n}} \right) \quad \text{as\,\,} x\to\infty,
\end{equation}
from which follows that  $du\in L^2(\R^n)_1$. Thus Definition~\ref{def:partialbalayage} can be used as stated, with just the additional requirement of mass balance
(\ref{intdud}) (the asymptotics (\ref{asymptoticsu}) need not be required explicitly).
These assumptions ensure partial integration without boundary contributions, as in (\ref{partialintegration}), and then existence and uniqueness of partial balayage
follow (assuming (\ref{sigmanegative})). Also (\ref{complementarity}) follows, 
after normalization of the additive constant in $u$ as in the beginning of Section~\ref{sec:variational}. 

Most of previous treatments of partial balayage in Euclidean space have been based either on Dirichlet boundary conditions in bounded domains (or at least domains
admitting an ordinary  Green's function), or else on full space $\R^n$ with the simplifying assumption that $\sigma_-$ is so big that the function $u$
automatically vanishes in a full neighbourhood of infinity.


\section{Simple examples of partial balayage}\label{sec:examples}

\subsection{A one-dimensional example}\label{sec:onedimension}

Even though the one dimensional case is not of primary interest, it gives a possibility to construct simple examples
and to build up the intuition. There is only one (up to diffeomorphisms) closed manifold of dimension one,
and this can be represented by the unit circle $S^1$, or by $\R/ \Z$. 
Using the latter,  functions, currents (etc.) on
$M$ get represented by periodic functions (etc.) on $\R$, or as the corresponding objects defined on the single period interval $[0,1)$
in such a way that they have good periodic extensions. The Riemannian metric will be $ds^2=dx^2$,
where $x$ is the coordinate on $\R$. 

The situation in dimension one differs from all higher dimensions in that all charge distributions have finite energy
and all potentials are continuous functions. Thus partial balayage always exists whenever (\ref{sigmanegative}) holds. 
The point with the example below is partly just to illustrate the general theory by computing all functions involved, but it is also good to see
the difference between one dimension and two (and higher) dimensions by comparing it with the example in Section~\ref{sec:singular}.

Representing $M$ by the single period interval $[0,1)$ we shall compute ${\rm Bal(\sigma,0)}$ with
\begin{equation}\label{sigmaab}
\sigma= \delta_a-2\delta_b,
\end{equation}
where $0\leq a<b<1$. 
Since ${\rm vol^1}(M)=1$ and $\int_M \sigma=-1$, the parameter $\textsf{t}$ in (\ref{t}) is $\textsf{t}=1$. 
Therefore, the equation (\ref{psirelaxed}) for the potential $\psi$ becomes $\psi''=\delta_a-2\delta_{b}+1$.
Integrating twice, taking into account that $\psi$ must be extendable to a periodic function without this causing extra contributions to
$\psi''$, gives, for $0\leq x<1$, 
$$
\psi(x)= \frac{1}{2}\,|x-a|-  |x-b| + \frac{1}{2}\, x^2  +(a-2b)x+C,
$$
where $C=\frac{1}{12}-\frac{a^2}{2}+b^2$ if the normalization (\ref{psivol}) is imposed.

In order to compute $u$ and $v$ we must know the outcome of the balayage process.
But there is actually not much choice, the mass at $a$ must go into the only available hole, at $b$. This gives
\begin{equation}\label{nostructure}
{\rm Bal}({\delta_a-2\delta_b,0})= -\delta_b,
\end{equation}
from which we easily get $u$ and $v$: they have to satisfy $u''=-\delta_a+\delta_{b}$ and $v''=-\delta_{b}+1$,
and integrating these equations twice taking into account periodicity constraints gives, with $C$ as above,
$$
\begin{cases}
\displaystyle u(x)=- \frac{1}{2}|x-a|+  \frac{1}{2}|x-b|+(b-a)x+(b-a)\left(\frac{1}{2}-b\right), \\[.4cm]
\displaystyle v(x)=-  \frac{1}{2}|x-b| + \frac{1}{2} \, x^2 -bx+ (b-a)\left(\frac{1}{2}-b\right)+C.
\end{cases}
$$
In the notations of Section~\ref{sec:structure} we obtain, within the period $[0,1)$,
$$
\omega=[0,1)\setminus \{a\},\quad \Omega=[0,1),
$$
$$
\mu=\delta_b, \quad \sigma_ -=2\delta_b.
$$
In particular we see that the structure formula (\ref{structure}) does not hold in this case.


\subsection{A singular case on the sphere}\label{sec:singular}

Here we take $M=S^2$, the unit sphere in $\R^3$, and with the Riemannian metric inherited from
$\R^3$. This means that in standard spherical coordinates $(\theta,\varphi)$, with
$0\leq \theta\leq \pi$, $0\leq \varphi <2\pi$, the metric is given by
$$
ds^2= d\theta^2 + \sin^2\theta \,d\varphi^2.
$$
The volume (or area) form is
$$
{\rm vol}^2=\sin\theta \,d\theta \wedge d\varphi.
$$

We shall try to perform the partial balayage ${\rm Bal(\sigma,0)}$, for a choice of $\sigma$ of the same kind as in the previous example (Section~\ref{sec:onedimension}), namely
\begin{equation}\label{sigmaNS}
\sigma= \delta_N-2\delta_S.
\end{equation}
Here $N$ and $S$ denote the north and south poles, given in spherical coordinates by $\theta=0$ for $N$, $\theta=\pi$ for $S$,
and with $\varphi$ being indeterminate in both cases. The energy of $\sigma$ is infinite, both at $N$ and at $S$.
Now ${\rm vol}^2(M)=4\pi$, $\mathsf{t}= -\mathsf{m}(\sigma) = 1/4\pi$, and the equation for $\psi$ becomes
$$
d*d\psi =\delta_N -2\delta_S + \frac{1}{4\pi} \, {\rm vol}^2.
$$
This can be solved explicitly, and the result is the potential
\begin{equation}\label{psisingular}
\psi= \frac{1}{4\pi} \left( \log \sin^2 \frac{\theta}{2} -2 \log \cos^2 \frac{\theta}{2}-1 \right),
\end{equation}
where the additive level is adjusted so that (\ref{psivol}) holds. 

If everything were as in the one-dimensional case above, then we would have
$$
u= \frac{1}{4\pi} \left( -\log \sin^2 \frac{\theta}{2} + \log \cos^2 \frac{\theta}{2} \right) ,
$$
$$
v= -\frac{1}{4\pi} \left( \log \cos^2 \frac{\theta}{2}+1 \right) .
$$
However, these potentials do not have the right properties, for example $u$ is not bounded from below, so it is impossible to make it
satisfy $u\geq 0$, even after adjustment of constants. Hence $v$ does not satisfy $v\geq \psi$. 

In fact, in the present case there is no function $v$ whatsoever which satisfies the requirements in Definition~\ref{def:partialbalayage2} in the sense that it satisfies
$v\geq \psi$ and $\Delta v\leq \textsf{t}$. For if $v$ is to be as big as $\psi$ at $S$, then $\Delta v$ has to have at least the negative contribution
$-2\delta_S$ at $S$, which then has to be compensated by the same amount of positive contribution somewhere else. And that is not possible
under the constraint $\Delta v\leq \textsf{t}$. The conclusion is that ${\rm Bal}({\delta_N-2\delta_S,0})$ does not exist.

Similarly, ${\rm Bal}({\delta_N-\delta_S,0})$ does not exist, despite what was said about the case $\int_M \sigma =0$ in the beginning of Section~\ref{sec:variational}.


\subsection{A mixed case on the sphere}\label{sec:mixed}

Here we shall soften up the previous example by introducing a volume term. We consider
\begin{equation}\label{sigmaNSvol}
\sigma= \delta_N-2\delta_S -\alpha \, {\rm vol}^2
\end{equation}
for suitable values of $\alpha>0$.
Even though none of $\sigma_{\pm}$ have finite energy we shall see that the results are better, provided $\alpha$ is large enough.

As a preparation we consider the more regular case 
\begin{equation}\label{sigmaNvol}
\sigma= \delta_N -\alpha \, {\rm vol}^2,
\end{equation}
for which $\sigma_-$ has finite energy.
When $0\leq \alpha<\frac{1}{4\pi}$ the partial balayage ${\rm Bal}(\sigma,0)$ of $\sigma$ in (\ref{sigmaNS}) does not exist because 
(\ref{sigmanegative}) is violated. So assume that
\begin{equation}\label{alphabound}
\alpha \geq  \frac{1}{4\pi}.
\end{equation}
Then everything is straightforward, for example $\psi$ has to satisfy 
$$
d*d\psi =\delta_N    -\alpha\, {\rm vol}^2+ \left( \alpha-\frac{1}{4\pi} \right) {\rm vol}^2 = \delta_N - \frac{1}{4\pi}\mathrm{vol}^2,
$$
which gives
$$
\psi= - G^{\delta_N}=\frac{1}{4\pi} \left( \log \sin^2 \frac{\theta}{2}+1 \right),
$$
and the balayage of the excess mass $\delta_N$ fills up the available ``hole'' (represented by $-\alpha \, {\rm vol}^2$)
in a circular neighbourhood of $N$. Precisely:
\begin{equation}\label{balNvol}
{\rm Bal}( \delta_N -\alpha \, {\rm vol}^2,0)=- \alpha \, {\rm vol}^2|_K,
\end{equation}
where $K$ is the unfilled part, defined by an equation $\theta_0\leq \theta\leq \pi$ with $\theta_0$ chosen so that $\alpha\, {\rm vol}^2(M\setminus K)=1$. 
The equation for $\theta_0$ becomes, more precisely,
$$
2\pi \alpha \left( 1-\cos \theta_0 \right) =1.
$$
A perhaps more intuitive way of writing (\ref{balNvol}) is 
$$
{\rm Bal}( \delta_N, \alpha \, {\rm vol}^2)= \alpha \, {\rm vol}^2|_\Omega,
$$
where $\Omega=S^2\setminus K$.

\begin{remark}
Translating the above formulae to the complex plane by stereographic projection from the north pole to the equatorial plane (to be identified with $\C$), so that 
$$
z=e^{i\varphi} \cot \frac{\theta}{2},
$$
gives
$$
ds^2=\frac{4|dz|^2}{(1+|z|^2)^2}, \quad
G^{\delta_N}=-\frac{1}{4\pi} \left( \log \frac{1}{1+|z|^2} + 1 \right).
$$
\blsquarehere
\end{remark}

Now, with the same lower bound (\ref{alphabound}) on $\alpha$ we return to (\ref{sigmaNSvol}). The hole becomes bigger in the
presence of the term $-2 \delta_S$, but it has infinite energy. Does this change anything? No, it turns out to that the new term
$-2\delta_S$ is just left untouched.  The equation for $\psi$ is the same as in Section~\ref{sec:singular}, because the value of
the parameter $\textsf{t}$ (see (\ref{t})) changes as a compensation for the volume term.  The new value is
$\textsf{t}= \frac{1}{4\pi}+\alpha$ and we have
$$
d*d\psi =\delta_N -2\delta_S -\alpha \, {\rm vol}^2+ \left( \frac{1}{4\pi}+\alpha \right) {\rm vol}^2.
$$
Thus $\psi$ is again given by (\ref{psisingular}). However, when  (\ref{alphabound}) holds it is now possible to find functions $v$ satisfying $v\geq \psi$ and $\Delta v\leq \textsf{t}$,
this due to $\textsf{t}$ now being bigger. In fact, 
$$
v= -\frac{1}{2\pi} \log \cos^2 \frac{\theta}{2}
$$
is a competing function, and it follows that ${\rm Bal}(\sigma,0)$ exists. 

As a remark, any function $v$ satisfying $v\geq \psi$ has, at $S$, a singularity at least as big as that of $\psi$, hence cannot have finite energy. Thus the cone $\mathcal{K}$ in
(\ref{K}) is empty, and we conclude that Definition~\ref{def:partialbalayage2} is in fact more general than what a definition based directly on Theorem~\ref{thm:partialbalayage1}
would have been.


\subsection{Examples on spheres in higher dimensions}
\label{sec:example-on-3-sphere}

Let us consider $M = S^3$, the unit sphere in $\R^4$, with the
inherited metric. For coordinates on $M$ we use the hyperspherical
coordinates $(\xi, \theta, \phi)$ defined by
\begin{align}
  \left\{
  \begin{array}{l}
    x_1 = \cos \xi, \\
    x_2 = \sin \xi \cos \theta, \\
    x_3 = \sin \xi \sin \theta \cos \phi, \\
    x_4 = \sin \xi \sin \theta \sin \phi,
  \end{array}
  \right.
\end{align}
where $0 \leq \xi \leq \pi$, $0 \leq \theta \leq \pi$ and
$0 \leq \phi \leq 2\pi$. The metric then becomes
\begin{align}
  ds^2 = d \xi^2 + \sin^2 \xi \, d \theta^2 + \sin^2 \xi \sin^2 \theta \, d \phi^2
\end{align}
in the hyperspherical coordinates, and so the volume form on $M$ is
\begin{align}
  {\rm vol}^3 = \sin^2 \xi \sin \theta \, d \xi \wedge d \theta \wedge d \phi.
\end{align}
Let $N = (1, 0, 0, 0)$ be the north pole on the sphere, corresponding
to $\xi = 0$ and both $\theta$ and $\phi$ indeterminate, and once more
consider the partial balayage of the charge distribution $\sigma$ defined
by
\begin{align}
  \sigma = \delta_N - \alpha \, {\rm vol}^3.
\end{align}
The volume of $M$ is
\begin{align}
  {\rm vol}^3(M) = \int_0^\pi \sin^2 \xi \, d\xi \int_0^\pi \sin \theta \, d\theta \int_0^{2\pi} d\phi = 2 \pi^2,
\end{align}
hence
\begin{align}
  \mathsf{t} = - \mathsf{m}(\sigma) = - \frac{1}{2\pi^2} \int \sigma = \alpha - \frac{1}{2\pi^2}.
\end{align}
We thus assume that $\alpha \geq 1/2\pi^2$ to ensure $\mathsf{t} \geq 0$.

In the hyperspherical coordinates the Laplacian becomes
\begin{align}
  \Delta f = \frac{1}{\sin^2 \xi} \left[ \frac{\partial}{\partial \xi} \left( \sin^2 \xi \frac{\partial f}{\partial \xi} \right) + \frac{1}{\sin \theta} \frac{\partial}{\partial \theta} \left( \sin \theta \frac{\partial f}{\partial \theta} \right) + \frac{1}{\sin^2 \theta} \frac{\partial^2 f}{\partial \phi^2} \right].
\end{align}
The equation for $\psi$ in the decomposition of $\sigma$ is
\begin{align}
  d * d \psi = \delta_N - \alpha \, \mathrm{vol}^3 + \left( \alpha -
  \frac{1}{2\pi^2} \right) \mathrm{vol}^3 = \delta_N -
  \frac{1}{2\pi^2} \, \mathrm{vol}^3.
\end{align}
For sake of finding $\psi$ we thus need to solve the equation
\begin{align}
  \Delta f = -\frac{1}{2 \pi^2},
\end{align}
in the region $0 < \xi < \pi$, which, assuming $f = f(\xi)$ for
symmetry, becomes
\begin{align}
  \frac{1}{\sin^2 \xi} \frac{\partial}{\partial \xi} \left( \sin^2 \xi
  \frac{\partial f}{\partial \xi}\right) = - \frac{1}{2\pi^2}
  \Rightarrow f(\xi) = \frac{\xi \cot \xi}{4 \pi^2} + A \cot \xi + B,
\end{align}
where $A$, $B$ are constants of integration. For the potential of
$\delta_N$ we require that the coefficient of $\xi^{-1}$ in the series
expansion of its potential around $\xi = 0$ is
$-|S^{3-1}|^{-1} = -1/4\pi$. It follows that
\begin{align}
  \psi = \frac{1}{4\pi^2} \left((\xi - \pi) \cot \xi + \frac{1}{2}\right),
\end{align}
where the additive constant is chosen so that \eqref{psivol} holds.
As it is easily seen that $\psi$ is bounded from above in $\xi$, it
follows from Example~\ref{ex:psiuppersemicontinuous} that
$\mathrm{Bal}(\sigma,0)$ exists. Precisely as in \eqref{balNvol} the
resulting balayage form is given by
\begin{align}
  \mathrm{Bal}(\delta_N - \alpha \, \mathrm{vol}^3, 0) = - \alpha \, \mathrm{vol}^3|_K,
\end{align}
with $K = S^3 \setminus \Omega$ and $\Omega$ is a ball around $N$ with
boundary ${\xi = \xi_0}$ for some constant $\xi_0$, determined
explicitly by the equation
\begin{align}
  \label{3-sphere-xi0-eq}
  \pi \alpha (2 \xi_0 - \sin (2\xi_0)) = 1.
\end{align}
Note that the function
$h(\xi) := \pi \alpha (2 \xi - \sin 2 \xi) - 1$ is continuous and
monotonically increasing on $[0, \pi]$, satisfies $h(0) = -1$, and
that $h(\pi) \geq 0$ holds if and only if $\alpha \geq 1/2\pi^2$,
ensuring that \eqref{3-sphere-xi0-eq} has a unique solution.

The above example can rather easily be generalized further to even
higher dimensions. For the $n$-sphere $M = S^n \subset \R^{n+1}$ we
can use the hyperspherical coordinates
$(\phi_1, \phi_2, \ldots, \phi_n)$ defined by $x_1 = \cos \phi_1$,
$x_j = \left( \prod_{k=1}^{j-1} \sin \phi_k \right) \cos \phi_j$ for
all $j = 2, 3, \ldots, n$, and $x_{n+1} = \prod_{k=1}^n \sin \phi_j$,
where $0 \leq \phi_1, \ldots, \phi_{n-1} \leq \pi$ and
$0 \leq \phi_n \leq 2\pi$. With $N = (1,0,0,\ldots,0)$ and
$\sigma = \delta_N - \alpha \, \mathrm{vol}^n$, using
$\alpha \geq 1/\mathrm{vol}^n(M)$ to ensure \eqref{sigmanegative},
one can find the potential for $\sigma$ by finding solutions
$\psi = \psi(\phi_1)$ to
\begin{align}
  \frac{1}{\sin^{n-1} (\phi_1)} \frac{\partial}{\partial \phi_1} \left(
  \sin^{n-1} (\phi_1) \frac{\partial \psi}{\partial \phi_1} \right) = -
  \frac{1}{\mathrm{vol}^n(M)}.
\end{align}
It turns out that the solutions are of the form
\begin{align}
  \psi(\phi_1) =& \frac{1}{\mathrm{vol}^n(M)} \int
  \frac{\cos(\phi_1)}{\sin^{n-1}(\phi_1)} \,_2F_1
  \left(\frac{1}{2},1-\frac{n}{2};\frac{3}{2}; \cos^2(\phi_1) \right)
  d \phi_1 \nonumber \\
  &+ A \cos(\phi_1) \,_2F_1 \left(\frac{1}{2}, \frac{n}{2};
  \frac{3}{2}; \cos^2(\phi_1) \right),
\end{align}
where ${}_2F_1$ is the Gaussian hypergeometric function,
\begin{align}
  {}_2F_1(a,b;c;z) = \sum_{n=0}^\infty \frac{(a)_n (b)_n}{(c)_n} \frac{z^n}{n!}
\end{align}
with $(k)_n = k (k + 1) \ldots (k + n - 1)$, $(k)_0 = 1$ the Pochhammer symbol, and $A$ is a constant of integration. The
resulting balayage $n$-form exists, and again has the form
\begin{align}
  \mathrm{Bal}(\delta_N - \alpha \, \mathrm{vol}^n, 0) = - \alpha \,
  \mathrm{vol}^n|_{S^n \setminus \Omega},
\end{align}
where $\Omega$ is the open ball around $N$ with boundary
$\{\phi_1 = \Phi\}$ for some constant $\Phi$ determined uniquely by
the equation
\begin{align}
  \alpha \, \mathrm{vol}^{n-1}(S^{n-1}) \int_0^\Phi \sin^{n-1}(\phi_1)\,d \phi_1 
  = 1.
\end{align}


\section{Examples of doubling technique}\label{sec:doubling}

\subsection{The double of a ball}\label{sec:double of ball}

In order to illustrate some matters in Section~\ref{sec:nonclosed}, let $M=B_R=B(0,R)$ be the open ball in $\R^n$ with radius $R$. 
We first construct the compact manifold $\hat{M}=M\cup\partial M\cup \tilde{M}$, the double of $M$. 
In $M$ we use the ordinary Euclidean metric
$$
ds^2=dx_1^2+\dots +dx_n^2, 
$$
and with $\tilde{x}_1, \dots, \tilde{x}_n$ the corresponding coordinates on $\tilde{M}$, the metric there will be 
$$
d\tilde{s}^2=d\tilde{x}_1^2+\dots +d\tilde{x}_n^2. 
$$
The general recipe for gluing these involve first choosing local coordinates, $y_1, \dots, y_n$, say, near the boundary
so that $\partial M$ corresponds to $y_n=0$ and $M$ to parts of the upper half space. For $\tilde{M}$ one does the same,
and then flips $\tilde{y}_n\mapsto -\tilde{y}_n$ before gluing.

Suitable coordinates in the ball case are spherical coordinates $r,\theta, \varphi, \dots$, which we in general may write as 
$(r,\omega)$, where $r>0$, $\omega\in S^{n-1}$. Then the Euclidean metric becomes (symbolically)
$$
ds^2=dr^2+ r^2 \, d\omega^2,
$$
for example $ds^2=dr^2+r^2 (d\theta^2+ \sin^2\theta \,d\varphi^2)$ in  dimension $n=3$.
For the local coordinates $y_1, \dots,y_n$ above we can choose $y_1,\dots,y_{n-1}$ to be various polar angles
(\eg $y_1=\theta$, $y_2=\varphi$ when $n=3$) and $y_n =R-r$. This renders the metric on the form
$$
ds^2= dy_n^2+(R-y_n)^2 \, d\omega(y_1,\dots,y_{n-1})^2 \quad (y_n\geq 0).
$$
The same expression is valid for $\tilde{M}$, with tilde on all coordinates, and then one allows $y_n$ to take negative values by
setting $y_n = -\tilde{y}_n\leq 0$. The resulting metric on (part of) $\hat{M}$ now becomes,  for $y_n$ in a full neighbourhood of $y_n=0$,
$$
ds^2= dy_n^2+ (R-|y_n|)^2 \, d\omega(y_1,\dots,y_{n-1})^2.
$$
Here one sees clearly that the metric tensor only becomes Lipschitz continuous, and this is the best one can achieve in general (see \cite{Friedrichs-1955},
\cite{Conner-1956}). In fact, the above choice of coordinates is already optimal in the sense  that, in terms of a general expression $ds^2=\sum g_{ij}dy_i\otimes dy_j$,
we have $g_{in}=0$ for all $i\ne n$. In cases when such mixed terms are present the coefficients $g_{ij}$ need not even be continuous
(this occurs with the coordinates for the ball chosen as in the example after Definition~1.1.1 in \cite{Schwarz-1995}).

In the somewhat trivial case of dimension $n=1$, the metric tensor actually becomes smooth (\eg in the above example there are
no polar angles), but already in dimension $n=2$ one has to treat Lipschitz continuous coefficients. The two dimensional case is
on the other hand favourable in the sense that the {\it conformal}  structure of the double remains smooth (for arbitrary $M$),  
which makes $\hat{M}$ become a true Riemann surface, the {\it Schottky double} of $M$.  

As an alternative to the coordinates $(\tilde{x}_1,\dots, \tilde{x}_n)$ or $(y_1,\dots, y_n)$ on $\tilde{M}$ one can use the original Cartesian coordinates
$(x_1,\dots, x_n)$ in the region outside $B_R$. Such a point $x=(x_1,\dots, x_n)$, hence with $|x|>R$,
 is then identified with a point  $\tilde{x}=(\tilde{x}_1,\dots, \tilde{x}_n)$ in $\tilde{M}$  (\ie  $|\tilde{x}|<R$) via the reflection map
\begin{equation}\label{transition}
(x_1,\dots, x_n)\mapsto(\tilde{x}_1,\dots, \tilde{x}_n)= \frac{R^2}{|x|^2} (x_1,\dots, x_n).
\end{equation}
These coordinates turn out to be quite useful and intuitive, for example the full Euclidean space $\R^n$ then represents all of $\hat{M}$ except for the ``north pole'',
$\tilde{0}\in \tilde{B}_R$.
Straight-forward computations give that, in these Cartesian coordinates,  the Riemannian metric on $\hat{M}\setminus\{\tilde{0}\}$ takes the form
\begin{equation}\label{doublemetric}
ds^2=
\begin{cases}
d{x}_1^2+\dots +d{x}_n^2 \quad & 0\leq |x|\leq R,\\
\displaystyle \frac{R^4}{|x|^4}(d{x}_1^2+\dots +d{x}_n^2) \quad &R<|x|<\infty.
\end{cases}
\end{equation}
Again we see that the metric is only Lipschitz continuous across $\partial M$. We also see that the metric is very small at infinity, in fact so small that the
one point compactification of $\R^n$ becomes a smooth manifold at infinity.

One may compare the above with the spherical metric on $\R^n$, \ie with the metric obtained from the standard metric on $S^n\subset\R^{n+1}$ by
stereographic projection from the north pole in  $S^n$ to $\R^n\cong \{x_{n+1}=0\}\subset\R^{n+1}$.
This is given  by
\begin{equation}\label{sphericalmetric}
ds_{\rm spherical}^2=\frac{4(d{x}_1^2+\dots +d{x}_n^2)}{(1+|x|^2)^2},
\end{equation}
hence is equally small at infinity, which certainly represents a smooth point of $S^{n}$.
For the spherical metric, the curvature is uniformly spread out over the manifold, while for the double of a ball the curvature is concentrated as a singular
distribution on $\partial M$. Indeed, the curvature tensor is an expression in the second order derivatives of the components of the metric tensor,
and these being just Lipschitz continuous means that the components of the curvature tensor will consist of measures sitting on $\partial M$. On the other hand $M$ and
$\tilde{M}$ are completely flat, but for topological reasons the manifold has to be curved somewhere. \vspace{.2cm}

As an example of function theory on the double we have
\begin{example}\label{ex:green}

If $u$ is a harmonic function, with some singularities, on a manifold with boundary, $M\cup \partial M$, and $u=0$
on $\partial M$, then $u$ can be extended to an odd function to the double $\hat{M}$ by setting $u(\tilde{x})=-u(x)$ at the 
point $\tilde{x}\in \tilde {M}$ opposite to $x\in M$, and with this extension $u$ remains harmonic in $\hat{M}$, except for its original and reflected singularities. 
A main example is the Green's function $g_M(\cdot,a)$ for $M$ with pole at $a\in M$,
which extends in this way to be harmonic on the double, with a corresponding counter-pole at the opposite point $\tilde{a}\in\tilde{M}$. 
This results in the following formula, which relates the Dirichlet Green's function to the Green's kernel (\ref{Gab}) for the double:
\begin{align*}
  g_M(x,a) &= \frac{1}{2}\mathcal{E}(\delta_x-\delta_{\tilde{x}},\delta_a-\delta_{\tilde{a}}) \\[.2cm]
           &= \frac{1}{2}[G(x,a)-G(x,\tilde{a})-G(\tilde{x},a)+G(\tilde{x},\tilde{a})].
\end{align*}
To prove the formula one just need to act by  $-d*d $ (with respect to $x$) on the right member, to see that it becomes $\delta_a$, and to check that the
right member vanishes when $x\in \partial M$.

\end{example}


\subsection{On equilibrium distributions}\label{sec:equilibrium}

The {\it classical equilibrium distribution} of a compact set $K\subset \R^n$ is the probability
measure $\mu$ on $K$ that minimizes the (unweighted) energy $\int U^\mu \, d\mu$ among all probability
measures on $K$. The corresponding equilibrium  potential $U^\mu$ is constant (quasi everywhere) on $K$
and behaves at infinity as $U^\mu (x)= \mathcal{O} (|x|^{2-n})$ ($n\geq 3$), $U^\mu(x)=-\frac{1}{2\pi}\log |x|+\mathcal{O}(1)$ ($n=2$).
 
If $\gamma$ is the constant value of $U^\mu$ on $K$, then the function 
$$
V(x)=\gamma -U^\mu(x)
$$
coincides, in the case $n=2$, with the Green's function $g_\Omega(x,\infty)$ of $\Omega=(\R^2\cup\{\infty\})\setminus K$.
Here $\R^2\cup\{\infty\}=S^2$ is the Riemann sphere with its usual conformal structure, and in two dimensions one 
need not specify the metric in order to define harmonic functions, like the Green's function.
It follows, as is well-known, that
the equilibrium distribution $\mu$ of $K$ coincides with the harmonic measure, $\nu=-d*g_\Omega(\cdot,\infty)$,  of the complementary domain
with respect to infinity (equivalently, with classical balayage of the point mass $\delta_\infty$ to $\partial\Omega$).

If one wishes something similar in higher dimension then one must first of all compactify $\R^n$ when $n\geq 3$, and then choose a Riemannian metric
on it. Compactification to a sphere $\R^n \cup \{\infty\}=S^{n}$ with its spherical metric, which in $\R^n$ becomes (\ref{sphericalmetric}), does not work, because
$V$ is simply not harmonic with respect to this metric. 

Another possibility is to choose a large ball $B_R$, which contains $K$, and then compactify by 
completing $B_R$ to the double $\hat{B}_R$. This has the advantage that the original metric in $B_R$ is kept unchanged.  
Again, this works well in two dimensions. Indeed, the Euclidean, the spherical and the metric of the double are all conformally equivalent,
hence the choice does not matter when extending harmonic functions.
However, in higher dimensions it does not work perfectly well. One could say that the difference compared to the spherical metric is that all curvature now
is concentrated to $\partial B_R$, and when trying to extend $V|_{B_R}$ harmonically to $\hat{B}_R$, with a necessary pole at the  
``point of infinity'' $\tilde{0}$, one gets a distributional contribution to $d*dV$ on $\partial B_R$. 

It is in fact easy to check this statement, because the only way to make such a continuation of $V$ is
to fold the original $V|_{\R^n\setminus {B}_R}$ over $\partial{B}_R$ by means of the Kelvin transform \cite{Helms-1969, Doob-1984, Armitage-Gardiner-2001}  
and then possibly add a harmonic function in $\tilde{B}_R\setminus \{\tilde{0}\}$ which vanishes on $\partial B_R$. The latter function must be of the form $A(|\tilde{x}|^{2-n}-R^{2-n})$ 
for some $A$, but no matter how one chooses $A$ there will be a jump of the normal derivative of $V$ on $\partial B_R$,
\ie there will be a distributional contribution to $d*dV$ on $\partial B_R$. (We omit the computational details.)

For {\it weighted equilibrium distributions}, the choice of Riemannian metric  matters also in two dimensions, because the volume form is involved.
We recall from Section~\ref{sec:weighted-eq-confs} and \cite{Roos-2015} the connection between partial balayage and weighted
equilibrium distributions: if $Q$ is a potential, bounded from below, on a compact manifold $M$, and we let, for any $\textsf{t}>0$,
\begin{equation}\label{sigmaQ}
\sigma_{\textsf{t}}= -d*dQ - \textsf{t}\,{\rm vol}^n,
\end{equation}
then ${\rm Bal}(\sigma_{\mathsf{t}},0)$ exists and relates to the $\mathsf{t}$-equilibrium measure $\mu_{Q,\mathsf{t}}$ for $Q$ by
$$
\mu_{Q,\mathsf{t}}+{\rm Bal}(\sigma_{\mathsf{t}},0)=0.
$$ 
We have $\mu_{Q,\mathsf{t}} \geq 0$, $\mathsf{m}(\mu_{Q,\mathsf{t}}) = \mathsf{t}$, and
\begin{align*}
  &Q+G^{\mu_{Q,\mathsf{t}}}\geq c_{\rm Robin} \quad {\rm in}\,\, M,\\
  &Q+G^{\mu_{Q,\mathsf{t}}}= c_{\rm Robin} \quad {\rm on}\,\, {\rm supp\,}\mu_{Q,\mathsf{t}}
\end{align*}
for some constant $c_{\rm Robin}$. With $u=Q+G^{\mu_{Q,\mathsf{t}}}-c_{\rm Robin}$ this system is the same as (\ref{umuomega}),
together with the fact that $\mu_{Q,\mathsf{t}}$ and $-\sigma_{\mathsf{t}}$ have the same normalized net mass (namely $\mathsf{t}$).

Assume now that $Q$ is such that
\begin{equation}\label{ddQ}
d*dQ = \mathsf{s}\,{\rm vol}^n -\tau,
\end{equation}
where $\tau \geq 0$ (and hence $\mathsf{s}\geq 0$).  This form of $Q$ is common in applications, where then $\tau$ may consist of
finitely many point masses and the volume term means, in Euclidean settings, that $Q(x)$ contains a term $\alpha |x|^2$ with
$\alpha>0$.  In the case (\ref{sigmaQ}), (\ref{ddQ}) in general, $(\sigma_\mathsf{t})_-$ is the positive constant
$\mathsf{s}+\textsf{t}$ times ${\rm vol}^n$, hence the structure theorem (Theorem~\ref{thm:structure}) gives that
$$
\mu_{Q,\mathsf{t}}=(\mathsf{s}+\textsf{t}) \, {\rm vol}^n |_{{\rm supp\,}\mu_{Q,\mathsf{t}}}.
$$ 
We give below some more specific examples.

\begin{example}\label{ex:weightedeq_sphere}
Let $M=S^2=\C\cup\{\infty\}$ with the metric
$$
ds^2 = d\theta^2+ \sin^2\theta \,d\varphi^2= \frac{4|dz|^2}{(1+|z|^2)^2}.
$$
The Green's kernel (see (\ref{Gab})) is in complex coordinates given by
\begin{equation}\label{Gabsphere}
G(a,b)=G^{\delta_a}(b)=-\frac{1}{4\pi} \left( \log\frac{|a-b|^2}{(1+|a|^2)(1+|b|^2)}+1 \right )\quad (a,b\in\C).
\end{equation}
We then choose, for some $\alpha, \beta>0$, $a\in \C$,
$$
Q(z)=\alpha G(z,\infty)+ \beta G(z,a)
$$
$$
=-\frac{1}{4\pi} \left[ \,\alpha \log\frac{1}{1+|z|^2}+\beta \log\frac{|z-a|^2}{1+|z|^2}+\alpha +\beta-\beta \log(1+|a|^2) \right].
$$
This $Q$ is an attempt to imitate, in the spherical case, the Euclidean version
\begin{equation}\label{QEuclidean}
Q_{\rm Euclidean}(z)= \alpha|z|^2+\beta \log \frac{1}{|z-a|},
\end{equation}
used in \cite{Balogh-Harnad-2009, Roos-2015}, for example. (The names of the constants do not match perfectly.)

The above choice of $Q$ for $S^2$ gives
\begin{align*}
  \sigma_{\mathsf{t}} &= -d*dQ- \mathsf{t} \, {\rm vol}^2 =\alpha \delta_ \infty+\beta\delta_a - \left(\frac{\alpha+\beta}{4\pi}
                        + \mathsf{t} \right) {\rm vol}^2 \\
  &= \alpha \delta_ \infty+\beta\delta_a- \frac{(\alpha + \beta + 4\pi\mathsf{t})}{\pi(1+|z|^2)^2} \, dxdy,
\end{align*}
and $\mu_{Q,\mathsf{t}}=-{\rm Bal}(\sigma_{\mathsf{t}},0)$. On adding the constant multiple of the volume form which appears above, and using also the structure 
formula (\ref{purestructure}), one gets
$$
{\rm Bal}(\alpha \delta_ \infty+\beta\delta_a, \left(\frac{\alpha+\beta}{4\pi} + \mathsf{t} \right) {\rm vol}^2)= \left(
  \frac{\alpha+\beta}{4\pi} + \mathsf{t} \right) {\rm vol}^2|_{\Omega},
$$
where $\Omega$ can be viewed as two spherical disks, with centres $\infty$ and $a$ respectively, ``smashed'' together. By Theorem~\ref{thm:equivalentballs}
these disks are geodesic as well as harmonic balls with these centres. As a subset of $\C$, the spherical disk with centre $a$ will also be a Euclidean disk,
or possibly a half-plane, but as such a disk the centre will not be $a$. The spherical disk with centre $\infty$ will of course be the complement of an ordinary Euclidean disk.

For $\mu_{Q,\mathsf{t}}$ we get, from the above,
$$
\mu_{Q,\mathsf{t}}= \left( \frac{\alpha+\beta}{4\pi} + \mathsf{t} \right) {\rm vol}^2|_{S^2\setminus\Omega}.
$$
The point with this approach is that it gives a good intuition for what the support of $\mu_{Q,\mathsf{t}}$ looks like, namely
that it is the complement of some kind of quadrature domain \cite{Sakai-1982, Shapiro-1992, Gustafsson-Shapiro-2005}, in the present case a two point quadrature domain.  In general, good information is available on
topology, geometry and regularity of boundaries of quadrature domains.

\end{example}

\begin{example} 
We try to repeat the previous example in the case $M=\hat{B}_R$, \ie that $M$ is the double of a ball, in $n\geq 2$ dimensions, and with 
$$
Q(x)=\alpha G(x, \tilde{0}) + \beta G(x, a).
$$
Here $\tilde{0}\in \tilde{B}_R$ plays the role of point of infinity and $a\in B_R$. 

Clearly, $G(x, \tilde{0})$ will only depend on the radius $r=|x|$, and it shall satisfy
$$
-d*d G(\cdot,\tilde{0}) =\delta_{\tilde{0}}-\textsf{m}(\delta_{\tilde{0}}) \, {\rm vol}^n,
$$
where 
$$
\textsf{m}(\delta_{\tilde{0}}) =\frac{1}{{\rm vol}^{n}(\hat{B}_R)}= \frac{\Gamma(n/2+1)}{2 \pi^{n/2} R^n} 
$$
In dimension $n\geq 3$ this gives
\begin{align*}
G(x, \tilde{0}) &= a_n |x|^2 + \frac{b_n}{R^{n-2}}+c_n && x\in B_R,\\
G(\tilde{x}, \tilde{0}) &= a_n |\tilde{x}|^2 +\frac{b_n}{|\tilde{x}|^{n-2}} +c_n &&\tilde{x}\in \tilde{B}_R,
\end{align*} 
where $a_n=\frac{\Gamma(n/2+1)}{4n \pi^{n/2} R^n}$, $b_n=\frac{\Gamma(n/2 + 1)}{n (n-2) \pi^{n/2}}$ and with the constant $c_n$ adapted to achieve the normalization (\ref{Gvol}).
In fact, the above  expression are of the right type,
and the coefficients are chosen so that the two functions take the same value on $\partial B_R$ and their normal derivatives there
have equal absolute values but are of opposite signs. All this makes the combined function continuously differentiable across $\partial B_R$,
meaning that the Laplacian of it will have no distributional contributions on $\partial B_R$.

On modelling the above function in $\R^n\cup \{\infty\}$, with $\tilde{B}_R$ represented by $(\R^n\cup\{\infty\})\setminus\bar{B}_R$, 
and with the metric (\ref{doublemetric}), it becomes, in dimension $n \geq 3$,
$$
G(x,\infty)= 
\begin{cases}
a_n |x|^2 + {b_n}{R^{2-n}}+c_n, &|x|\leq R,\\[.2cm]
a_n {R^4}{|{x}|^{-2}} +{b_n}{R^{2(2-n)}}|x|^{n-2}+c_n, & |x|>R,
\end{cases}
$$ 
where now $\infty=\tilde{0}$ really becomes a point of infinity.

In dimension $n=2$ we get instead, denoting the variable by $z$,  
\begin{align*}
G(z, \tilde{0}) &= a_2 |z|^2 - b_2 \log R+c_2 && z\in B_R,\\
G(\tilde{z}, \tilde{0}) &= a_2 |\tilde{z}|^2 -{b_2}\log {|\tilde{z}|} +c_2 &&\tilde{z}\in \tilde{B}_R,
\end{align*} 
with $a_2=1/(8\pi R^2)$, $b_2 =1/2\pi$, hence in $\R^2\cup\{\infty\}=\C\cup\{\infty\}$,
$$
G(z,\infty)= 
\begin{cases}
a_2 |z|^2 - {b_2}\log R+c_2, &|z|< R,\\[.2cm]
a_2 {R^4}{|{z}|^{-2}} +{b_2}\log |z|-2b_2 \log R+c_2,& |z|>R.
\end{cases}
$$ 

It is possible to compute also the more general two point Green's function $G(z,a)$, with $a\in B_R$ say, in two dimensions. 
In fact, one need only to subtract $b_2\log |z-a|$ from the last expression for $G(z,\infty)$, this will add a pole of the right strength at $z=a$
and simultaneously kill the pole at $z=\infty$. Recall that $\log|z-a|$ is harmonic in all $\R^2\setminus\{a\}$ with the metric (\ref{doublemetric}), despite
this metric changing behaviour on $|z|=R$. The resulting Green's potential is, with an  $a$-dependent constant $d_2 = d_2(a)$,
$$
G(z,a)=
\begin{cases}
a_2{|z|^2}-b_2\log|z-a|+d_2,\quad &|z|<R,\\[.2cm]
\displaystyle a_2 {R^4}{|{z}|^{-2}} -b_2\log |z-a| +b_2\log \frac{|z|}{R}+d_2, \quad &|z|>R.
\end{cases}
$$
One can directly verify that (\ref{ddG}) holds (with $\omega=\delta_a$).

The above is exactly what we wanted to achieve, namely that the Green's potentials
$$
G^{\alpha\delta_{\infty}+\beta \delta_a}(z)= \alpha G(z, \infty)+\beta G(z,a)
$$ 
that we have constructed on the compact Riemannian manifold $\hat{B}_R$ represent, within $B_R$, exactly those Euclidean external potentials
which appear in (\ref{QEuclidean}).
One may view all this  as a way of extending, and regularizing at infinity, a basic background potential (\ref{QEuclidean})
in a similar way as was done in \cite{Roos-2015}. We emphasise, however, that we could arrange these matters perfectly well only in two dimensions.

\end{example}


\section{Examples of partial balayage in Euclidean balls}\label{sec:ball}

\subsection{Partial balayage with Dirichlet boundary conditions}\label{sec:Dirichlet}

We shall discuss in detail a specific balayage problem in a ball, with the aim of illustrating some subtleties and dependence on
boundary conditions in the theory of partial balayage on open manifolds.

Let $\eta$ denote hypersurface measure ${\rm vol}^{n-1}$ on the unit sphere $S^{n-1}=\partial B(0,1)$, and let $\rho$, $R$ be radii satisfying $0<\rho<1<R<\infty$. 
Our basic manifold will be $M=B(0,R)\subset \R^n$. Since we are in Euclidean space, ${\rm vol}^n$ is ordinary Lebesgue measure, and for simplicity we
suppress it from notation, \ie we represent, for example, absolutely continuous $n$-forms by their coefficients with respect to ${\rm vol}^n$. We shall study partial balayage 
${\rm Bal}(\sigma,0)$ with Dirichlet boundary conditions (see Section~\ref{sec:nonclosed}) of
$$
\sigma= t\eta -\chi_{B(0,\rho)}
$$ 
for various $t>0$, and eventually with $R\to\infty$. We recall that 
$$
{\rm Bal}(\sigma,0)=\sigma+d*du,
$$
where $u\in W^{1,2}_0(M)$ is the unique solution of the complementarity system (\ref{complementaritysystem}).

As all data are rotationally symmetric we have effectively a one dimensional problem, with the radius $r=|x|$
as independent variable. For general reasons, namely the structure formulas in Theorem~\ref{thm:structure}, the 
balayage will be of the form
\begin{equation}\label{balsigmachi}
{\rm Bal}(\sigma,0)=-\chi_{{B(0,s)}}
\end{equation}
for some radius $0\leq s\leq \rho$. This $s$ will depend on $R$, $t$ and $\rho$, but $\rho$ will be kept fixed all the time.
For $t$ we shall put an upper bound which guarantees that $s>0$ for all values of $R$, even in the limit $R\to\infty$. 
Such an upper bound is obtained from  (\ref{sigmastrictlynegative}) which, while not being a necessary assumption in the case
of Dirichlet boundary data, is sufficient for ensuring  the existence of a free boundary (represented in this example by $s>0$).
In the present notations (\ref{sigmastrictlynegative}) becomes
\begin{equation}\label{boundsont}
0<t<\frac{{\rm vol}^n(B(0,\rho))}{{\rm vol}^{n-1}(\partial B(0,1))} = \frac{\rho^n}{n},
\end{equation}
henceforth assumed. 

The above data means that the potential $u=u(r)$, which we extend by zero for $r>R$,  
shall be a continuous function in $0<r<\infty$ and satisfy the following additional requirements.
\begin{align*}\label{usystem}
  u(r)=0\qquad &0<r< s,\\
  [u'(r)]_{{\rm jump}}=0\qquad &r=s,\\
  u''(r)+\frac{n-1}{r}u'(r)=1\qquad &s<r<\rho,\\
  [u'(r)]_{{\rm jump}}=0\qquad &r=\rho,\\
  u''(r)+\frac{n-1}{r}u'(r)=0\qquad &\rho<r<1,\\
  [u'(r)]_{{\rm jump}}={-t}\qquad &r=1,\\
  u''(r)+\frac{n-1}{r}u'(r)=0\qquad &1<r<R,\\
  u(r)=0 \qquad & R\leq r<\infty.
\end{align*}
 Jumps are generally defined by
$$
[u'(r)]_{{\rm jump}}=\lim_{\varepsilon\searrow 0} \left( u'(r+\varepsilon)-u'(r-\varepsilon) \right).
$$
At $r=s$, representing the free boundary, as well as at $r=\rho$,  
also $u'$ (in addition to $u$) has to be continuous, while at $r=R$ this is not required.
Instead, the jump of $u'(r)$ at $r=R$ 
represents excessive mass moved  to the boundary of the manifold.

In the above system, $s$ is not known in advance, and for only one value of $s$ there exists a solution
$u$ of the system, a solution which then is unique. Each of the three differential equations is easily solvable, indeed the solutions
will be of the form
$$
u(r)=
\begin{cases}
Ar^{2-n}+B\quad &(n\ne 2),\\[.2cm]
A\log r + B\quad &(n=2)
\end{cases}
$$
in the two homogeneous cases, with an additional term $\frac{r^2}{2n}$ for the inhomogeneous case.
These general solutions then contain three  sets  of constants $\{A,B\}$ (as there are three differential equations),
and in addition we have the unknown $s$. So there are seven unknowns. The equations we have for these unknowns
are those which express continuity of $u$ at $r=s$, $r=\rho$, $r=1$, $r=R$ and the prescribed jumps of $u'$ at
$r=s$, $r=\rho$, $r=1$. So there are also seven equations. These are linear in the constants $\{A,B\}$, but nonlinear in $s$.

To give some details, let the constants be $\{A_j,B_j\}$, where $j=1,2,3$ for the three differential equations in the order they are written above.
In dimension $n\ne 2$ one may first express these constants in terms of $s$ as
\begin{align*}
A_1&=\frac{s^n}{n(n-2)},\\
B_1&=-\frac{s^2}{2(n-2)},\\
A_2&=\frac{s^n-\rho^n}{n(n-2)},\\
B_2&=- \frac{s^2-\rho^2}{2(n-2)} ,\\
A_3&=\frac{t}{n-2}+\frac{s^n-\rho^n}{n(n-2)},\\
B_3&=\frac{1}{R^{n-2}} \left( \frac{t}{n-2}+\frac{s^n-\rho^n}{n(n-2)} \right) ,
\end{align*}
which uses all requirements above except the continuity of $u(r)$ at $r=1$. That requirement gives
$$
A_2+B_2=A_3+B_3,
$$
which then becomes an equation for $s$, namely
\begin{equation}\label{Rtsrho}
 \left( 1+\frac{1}{R^{n-2}} \right) \left( \frac{t}{n-2} + \frac{s^n-\rho^n}{n(n-2)} \right) + \frac{s^n-\rho^n}{2n}=0.
\end{equation}

The main question is how much of the mass of $\sigma_+= t\eta$ goes to the outer
boundary $\partial B(0,R)$, in particular in the limit $R\to \infty$.  The density of this mass is, for finite $R$,
$$
-u'(R)= \frac{(n-2)A_3}{R^{n-1}},
$$
hence the total mass is
\begin{equation}\label{totalmass}
-\int_{\partial B(0,R)} *du =(n-2)|S^{n-1}| A_3.
\end{equation}
To pass to the limit $R\to\infty$, recall that  $\rho$ and $t$ are fixed and $s=s(R)$ stays in the interval $0<s(R)<\rho$. 
It then follows from (\ref{Rtsrho})  that $s(\infty)=\lim_{R\to \infty} s(R)$ is given by
$$
s(\infty)^n=\rho^n-2t,
$$
provided $n\geq 3$. When $n=1$, the first factor in (\ref{Rtsrho}) behaves in a different way, and one ends up with
$$
s(\infty)=\rho-t.
$$

Inserting the above expressions into previous equations gives $A_3=t/n$ when $n\geq 3$. Hence the amount of  mass disappearing at infinity is 
$$
-\lim_{R\to\infty}\int_{\partial B(0,R)} *du =\frac{(n-2)|S^{n-1}|}{n}t.
$$
This makes up the fraction ${(n-2)}/{n}$ of the total amount $\int \sigma_+=t|S^{n-1}|$ available for balayage. 
When $n=1$ one gets instead $A_3=0$ (in the limit $R\to\infty$), so no mass  disappears at infinity in this case.

When $n=2$ one gets slightly different equations, which result in
\begin{align*}
A_1&=-\frac{s^2}{2},\\
B_1&=\frac{s^2}{2}\log s-\frac{s^2}{4},\\
A_2&=\frac{\rho^2-s^2}{2},\\
B_2&=\frac{s^2}{2}\log s-\frac{\rho^2}{2}\log \rho+\frac{\rho^2-s^2}{4},\\
A_3&=\frac{\rho^2-s^2}{2}-t,\\
B_3&= \left( t-\frac{\rho^2-s^2}{2} \right) \log R.
\end{align*}
Here we have used all equations in the system of jump conditions except the one which expresses continuity
of $u(r)$ at $r=1$. That equation gives $B_2=B_3$. 
Since $B_2$ is a  bounded function of $s$ and $\rho$ (with $0<s<\rho<1$), so is $B_3$,
hence it follows that the first factor in $B_3$ has to go to zero  as $R\to\infty$. This gives
$$
s(\infty)^2=\rho^2-2t.
$$
Thus $A_3=0$ in the limit $R\to\infty$, which  gives
$$
-\lim_{R\to\infty}\int_{\partial B(0,R)} *du =0,
$$
\ie that no mass is lost in the limit.

As a summary of the above example, dimensions $n=1,2$ are special in the sense that net mass is preserved under the balayage process
$\sigma\mapsto {\rm Bal}(\sigma ,0)$ in $\R^n$, when this is treated as a limiting case of balayage with Dirichlet
boundary conditions in bounded domains, while in dimension $n\geq 3$, a fraction $(n-2)/n$ of the available mass  disappears at infinity.


\subsection{Other boundary conditions}\label{sec:Neumann}

In the beginning of Section~\ref{sec:nonclosed}, three other types of boundary conditions were mentioned, besides vanishing Dirichlet data,
namely relaxed and hydrodynamic Dirichlet data, and Neumann data. With any of these conditions no mass disappears, even in the case of finite $R$. 
The assumption (\ref{sigmanegative}) is now necessary, and the mass balance gives immediately that $s$ in (\ref{balsigmachi}) is given by 
$$
s^n =\rho^n-nt, 
$$
for arbitrary values of $R>1$ (and $n\geq 1$). 

The vanishing Dirichlet data $u(r)=0$ for $R\leq r<\infty$  will in all three cases
be replaced by vanishing Neumann data $u'(R)=0$. Then $u$ is naturally extended by $u(r)=u(R)$ for $R\leq r<\infty$.
The so obtained solution $u(r)$ will then not depend on $R$, and it will also give the solution of the balayage
problem in entire space $M=\R^n$.


\subsection{Excess mass in Dirichlet case}\label{sec:excess}

We here give upper bounds for how much mass in general is moved to the boundary
in case of partial balayage in a bounded domain $M\subset \R^n$ with Dirichlet boundary conditions.
Thus we consider
$$
{\rm Bal}(\sigma,0)=\sigma + d*du,
$$
where $u\in W^{1,2}_0(M)$ minimizes $\int_M du\wedge*du$ under the constraint $\sigma+d*du\leq 0$. The notation ${\rm Bal}(\sigma,0)$
stands only for the mass within $M$, and the excess mass is represented by (minus) the $(n-1)$-form $*du$ on $\partial M$.  
Taking this into account gives full mass balance:
$$
\int_M \sigma= \int_M {\rm Bal}(\sigma,0)- \int_{\partial M} *du.
$$

The main result in this section will be a confirmation in general of what we saw in the example in Section~\ref{sec:Dirichlet}, namely that 
in dimension one and two, the excess mass disappears in the limit as $M$ grows to $\R^n$. It is easy to see, using properties as in Remark~\ref{rem:properties},
that the total excess mass decreases whenever $M$ is enlarged, so it will be enough to discuss the case that $M$ is a ball, say  $M=B_R=B(0,R)$.   
 
\begin{theorem}\label{thm:excess}
Given any charge distribution $\sigma$ with compact support in $\R^n$, say ${\rm supp\,}\sigma\subset B_\rho$, and satisfying (\ref{sigmastrictlynegative}),
consider partial balayage of $\sigma$ with Dirichlet boundary conditions in balls $B_R$ with $R>\rho$.
Writing the result as 
$$
\nu_R={\rm Bal}_R(\sigma,0)=\sigma + d*du_R,
$$ 
and setting
\begin{equation}\label{qR}
q_R=-\int_{\partial B_R} *du_R,
\end{equation}
we have the estimates
\begin{align}\label{qRn}
q_R^2&\leq 
\frac{(n-2)|S^{n-1}|(R\rho)^{n-2}}{R^{n-2}-\rho^{n-2}} \,\mathcal{E}(\tilde{\nu}-\sigma) \quad &&{\rm when}\,\,n\ne 2,\\
q_R^2&\leq \frac{2\pi}{\log R-\log\rho}\,\mathcal{E}(\tilde{\nu}-\sigma) \quad &&{\rm when}\,\,n=2.
\end{align}
Here
\begin{equation}\label{nutilde}
\tilde{\nu}= \left( \frac{\int\sigma_+}{\int\sigma_-}-1 \right) \,\sigma_-,
\end{equation}
which is independent of $R$, has finite energy and satisfies $\tilde{\nu}\leq 0$, $\int\tilde{\nu}=\int\sigma$. 
It follows that $q_R\to  0$ as $R\to \infty$ when $n=1$ or $n=2$.
\end{theorem}

\begin{proof}
By general  properties of partial balayage, for example the structure formulas in Theorem~\ref{thm:structure}, 
the function $u_R\in W^{1,2}_0(B_R)$  is harmonic in  $B_R\setminus\overline{B_\rho}$, which is a subset of $B_R\setminus {\rm supp\,}\sigma$, and it satisfies (\ref{qR}).
Ignoring everything else  we now look for that function  $U=U_R$ which minimizes $\int_M dU\wedge*dU$ among all functions
$U\in W^{1,2}_0(B_R)$ having these two properties (with $q_R$ kept fixed). Then obviously 
\begin{equation}\label{Uu}
\int_{B_R} dU_R\wedge*dU_R\leq \int_{B_R} du_R\wedge*du_R.
\end{equation}

The point here is that $U_R$ can be easily computed, because it will be proportional to the conductor potential 
associated to the pair of conductors $\partial B_\rho$ and $\partial B_R$.
In fact, the variational formulation of the minimization problem for $U_R$ gives that $U_R$ has to be constant in  
$B_\rho$, besides being harmonic in $B_R\setminus\overline{B_\rho}$. 
Straight-forward calculations then give the expressions, valid in the harmonic region $\rho<r<R$,
$$
U_R(r)=
\begin{cases}
\displaystyle \frac{q_R}{(n-2)|S^{n-1}|} \left( \frac{1}{r^{n-2}}-\frac{1}{R^{n-2}} \right) \quad &(n\ne 2),\\[.4cm]
\displaystyle \frac{q_R}{2\pi}(\log R-\log r)\quad &(n= 2).
\end{cases}
$$
By this
\begin{align}
  \int_{B_R} dU_R\wedge*dU_R &= \int_{B_R\setminus B_\rho} dU_R\wedge*dU_R \nonumber \\
                             & =\begin{cases}
                               \displaystyle \frac{q_R^2 }{(n-2)|S^{n-1}|} \left( \frac{1}{\rho^{n-2}}-\frac{1}{R^{n-2}} \right) \quad &(n\ne 2),\\[.4cm]
                               \displaystyle \frac{q_R^2}{2\pi}(\log R-\log \rho)\quad &(n=2).
                             \end{cases} \label{UU}
\end{align}

Recall next that
$$
\int_{B_R} du_R\wedge*du_R=\mathcal{E}(\nu_R-\sigma).
$$
Besides the lower bound (\ref{Uu}), (\ref{UU}) for this quantity there are some fairly obvious upper bounds. Among the competitors $\nu$ for 
minimizing $\mathcal{E}(\nu-\sigma)$ one can simply choose  any $\nu$ which rearranges $\sigma$ directly, and independently of $R$, by putting  the 
mass $\sigma_+$ into parts of the available holes represented by $\sigma_-$. Recall that we have assumed that (\ref{sigmastrictlynegative}) holds and that
$\sigma$ has finite energy. One such choice of $\nu$ is given by $\tilde{\nu}$ in (\ref{nutilde}). Thus
$$
\mathcal{E}(\nu_R-\sigma)\leq \mathcal{E}(\tilde{\nu}-\sigma)<\infty.
$$
Combining this with (\ref{Uu}) and (\ref{UU}) gives the assertions of the proposition.

\end{proof}


\bibliography{bibliography_gbjorn}


\noindent Bj\"orn Gustafsson \\[.1cm]
Department of Mathematics \\
KTH \\
SE-100 44 Stockholm \\
Sweden \\[.1cm]
e-mail: \texttt{gbjorn@kth.se} \\

\noindent Joakim Roos \\[.1cm]
Department of Mathematics \\
KTH \\
SE-100 44 Stockholm \\
Sweden \\[.1cm]
e-mail: \texttt{joakimrs@math.kth.se}

\end{document}